\documentclass[12pt, ams fonts]{amsart}

\usepackage{color}
\usepackage{stmaryrd}
\usepackage{tikz-cd}
\usepackage{mathtools}
\usepackage{amssymb,amsmath}
\usepackage{amsfonts}
\usepackage{graphicx}
\usepackage{tikz-cd}
\usepackage{graphicx}
\usepackage{yhmath}
\usepackage{mathdots}
\usepackage{MnSymbol}

\usepackage[vcentermath,enableskew]{youngtab}

\input xy
\xyoption{all}

\font\teneufm=eufm10 \font\seveneufm=eufm7 \font\fiveeufm=eufm5 
\newfam\eufmfam
\textfont\eufmfam=\teneufm \scriptfont\eufmfam=\seveneufm
\scriptscriptfont\eufmfam=\fiveeufm

\newcommand{\C}{\mathbb{C}}

\newcommand{\Z}{\mathbb{Z}}

\newcommand{\g}{\mathfrak{g}}

 \DeclareMathOperator{\Aut}{Aut}

\DeclareMathOperator{\re}{Re}

\DeclareMathOperator{\diag}{diag}
\DeclareMathOperator{\Mat}{Mat}
\DeclareMathOperator{\GL}{GL}

\DeclareMathOperator{\I}{I}

\DeclareMathOperator{\End}{End}
\DeclareMathOperator{\Orb}{Orb}

\numberwithin{equation}{section}

\newtheorem{definition}{Definition}[section]

\newtheorem{example}[definition]{Example}

\theoremstyle{remark}

\newtheorem{remark}[definition]{Remark} 

\theoremstyle{plain} 

\newtheorem{theorem}[definition]{Theorem}
\newtheorem{lemma}[definition]{Lemma}
\newtheorem{corollary}[definition]{Corollary}
\newtheorem{proposition}[definition]{Proposition}

\def\Z{\mathbb Z}
\def\C{\mathbb C}

\begin{document}

\title{A family of simple $U(\mathfrak{h})$-free modules of rank 2 over $\mathfrak{sl} (2)$}
\author[D. Grantcharov]{Dimitar Grantcharov}
	\address{D. Grantcharov: Department of Mathematics, University of Texas at Arlington}
	\email{grandim@uta.edu}	
\author[K. Nguyen]{Khoa Nguyen}
	\address{K. Nguyen: Department of Mathematics and Statistics, Queen's University, Kingston, ON K7L 3N6, Canada}
	\email{k.nguyen@queensu.ca}	
\author[K. Zhao]{Kaiming Zhao}
	\address{K. Zhao: Department of Mathematics, Wilfrid Laurier University, Waterloo, ON N2L 3C5, Canada, and School of Mathematical Science, Hebei Normal University, Shijiazhuang, Hebei 050024, China.}
	\email{kzhao@wlu.ca}
         

\maketitle

\begin{abstract}
We study simple $\mathfrak{sl}(2)$-modules over $\C$ that are free of finite rank as $U(\mathfrak h)$-modules, where $\mathfrak h$ is a Cartan subalgebra of $\mathfrak{sl}(2)$. Our main result is an explicit classification of the scalar-type simple modules of rank $2$. We also give a criterion for when two such modules are isomorphic. Both the classification and the isomorphism problem reduce to twisted conjugacy classes in $\GL_2(\C[h])$ and rely on Cohn's standard form.



\end{abstract}

\section{Introduction}
The Lie algebra  $\mathfrak{sl}(2)$ is a classical object and continues to serve as a basic point of reference in the representation theory of complex Lie algebras. The category of
generalized weight $\mathfrak{sl}(2)$-modules---i.e.\ modules on which a Cartan subalgebra acts locally
finitely, and which contains the weight modules as a full subcategory---has been described in detail; see,
for example, \cite{Ma} and the references therein.  A key feature is that every such module decomposes as
a direct sum of generalized $\mathfrak h$-weight spaces (equivalently, $\mathfrak h$-locally finite
submodules), for a fixed Cartan subalgebra $\mathfrak h=\C h$ of $\mathfrak{sl}(2)$.

In contrast, outside the generalized weight setting the representation theory becomes much less explicit,
even for simple modules. Block’s foundational work~\cite{Bl} shows that simple $\mathfrak{sl}(2)$-modules
are parameterized by similarity classes of irreducible elements in a certain Euclidean (Ore) skew algebra
$R$ of differential-operator type. In particular, the non-weight simple modules arise as cyclic left
quotients $R/Ra$ for such irreducible elements $a\in R$. While this gives a conceptual
parametrization, irreducible elements of $R$ are difficult to describe in concrete terms, so the result
does not by itself provide a convenient realization of explicit families of simple non-weight modules.

A natural and broad class of non-weight modules is obtained by imposing freeness over $U(\mathfrak{h})$.
Let $\mathcal{M}$ be the full subcategory of $U(\mathfrak{sl}(2))$-modules whose restriction to
$U(\mathfrak{h})$ is free of finite rank; equivalently, for $M\in\mathcal{M}$ one has
$M\simeq U(\mathfrak{h})^{\oplus n}=\mathbb{C}[h]^{\oplus n}$ as a $U(\mathfrak{h})$-module for some $n\ge 1$.
These modules can be viewed as complementary to the generalized weight setting: rather than decomposing
into $\mathfrak{h}$-(generalized) weight spaces, they are free over $U(\mathfrak{h})$.
The rank-one case is completely understood~\cite{Nil1}, while for higher rank important but partial results
and constructions are available; see, e.g.,~\cite{BSh,FLM, LN, GN, GN2, Me, MP}.

In this paper we initiate the study of rank-$2$ objects $M$ in $\mathcal M$, with the goal of eventually
obtaining an explicit classification of the simple modules in this rank.  Choosing a $\C[h]$-basis
identifies $M$ with $\C[h]^{\oplus 2}$, so that $h$ acts by multiplication.  The generators $e$ and $f$
then act by explicit difference operators: their actions are encoded by matrices in
$\Mat_2(\C[h])$ composed with the shift automorphisms $\sigma$ and $\sigma^{-1}$, respectively, where
$\sigma\big(p(h)\big)=p(h-1)$ (see Theorem~\ref{U(h)-free-realization}).  A convenient parameterization of such
modules uses three pieces of data: a central character parameter $\alpha\in\C$, an ordered partition
${\bf a}=(a_1,a_2,a_3)$ of $2$, and a matrix $K(h)\in\GL_2(\C[h])$. Here $\bf a$ and $K(h)$ are determined from the Smith normal form $P_{\alpha,{\bf a}}$ 
of the polynomial matrix  encoding the action of $f$.

In the present paper we complete the first part of the rank-$2$ project, namely the case of
\emph{scalar-type} modules, i.e.\ those for which the Smith normal form $P_{\alpha,{\bf a}}$ is a scalar matrix,
equivalently ${\bf a}\in\{(2,0,0),(0,2,0),(0,0,2)\}$.
More precisely, we give a complete classification and an explicit realization of all \emph{simple}
scalar-type modules in $\mathcal M$.
The classification relies on the uniqueness of Cohn's standard form for matrices in $\GL_2(\C[h])$
\cite{Co}. For explicit computations we relate Cohn's standard form to another canonical factorization, which we call 
the \emph{LQ form}. It is obtained by applying the Euclidean algorithm to the first row of
$K(h)\in\GL_2(\C[h])$, and expresses $K(h)$ as a lower triangular front factor followed by a product of the
matrices $E(\cdot)$.
To our knowledge, this is the first paper to use decompositions in $\GL_2(\C[h])$ systematically in the rank-$2$
$U(\mathfrak h)$-free setting. The rank-$2$ case already presents a surprisingly rich collection of simples, and
extending these results to $\GL_n(\C[h])$ for $n\ge3$ poses a substantial challenge.

The isomorphism classes are governed by a natural twisted conjugacy relation on
$\GL_2(\C[h])$ (cf.\  Definition~\ref{sigmadef}): two matrices $A$ and $B$ are $\sigma^{-1}$-similar if
\[
B=P(h)^{-1}A\,P(h+1)\qquad\text{for some }P(h)\in\GL_2(\C[h]).
\]
Equivalently, these are the twisted conjugacy classes for the $\Z$-action on
$\GL_2(\C[h])$ induced by $\sigma^{-1}$. One can also view the set of $\sigma^{-1}$-similarity classes
as a non-abelian cohomology set $H^{1}(\Z,\GL_2(\C[h]))$ for this action. The main results are
summarized below.

\smallskip

\noindent{\bf (1) Matrices $K(h)$ for modules with no rank-$1$ submodules.}
The classification easily reduces to describing those $K(h)\in\GL_2(\C[h])$ for which the associated scalar-type module
contains no rank-$1$ submodules. Set
\[
\mathcal S
:=\GL_2(\C[h])\setminus\bigcup_{a,b\in\C^*}\Orb_{\sigma^{-1}}\bigl(\diag(a,b)\bigr),
\]
the complement of the $\sigma^{-1}$-similarity orbits of constant diagonal matrices.  By
Lemma~\ref{lemsiginv2} and Lemma~\ref{lemsiginv}, for scalar-type modules the condition $K(h)\in\mathcal S$
is equivalent to the absence of rank-$1$ submodules.  Using Cohn's standard form in $\GL_2(\C[h])$
\cite{Co}, built from the matrices
\[
E(u(h))=\begin{pmatrix}u(h)&1\\-1&0\end{pmatrix},
\]
we prove (Theorem~\ref{mainstructureS}) that $K(h)\in\mathcal S$ if and only if $K(h)$ is $\sigma^{-1}$-similar
to a matrix of the form
\[
\mathbf{E}_{(\beta_1,\beta_2)}(v_1,\dots,v_m)
=\diag(\beta_1,\beta_2)\,E(v_1)\cdots E(v_m),
\qquad \beta_1,\beta_2\in\C^*,\ m\ge1,\ v_i\in\C[h]\setminus\C.
\]
We also give a complete and explicit list of the representatives in standard form for matrices outside $\mathcal S$
(Theorem~\ref{explicitdescription}).

\smallskip

\noindent{\bf (2) Simplicity and isomorphisms for scalar-type rank $2$.}
Every scalar-type rank-$2$ object in $\mathcal M$ is of the form
\[
M\left(\alpha,{\bf a},K(h)\right),
\qquad
{\bf a}\in\{(2,0,0),(0,2,0),(0,0,2)\},\ \alpha\in\C,\ K(h)\in\GL_2(\C[h]).
\]
The simplicity criterion (Theorem~\ref{newfamily}) shows that $M(\alpha,{\bf a},K(h))$ is simple
if and only if $K(h)\in\mathcal S$, with one additional restriction in the case ${\bf a}=(0,2,0)$:
\[
\alpha\notin 1+\tfrac12\Z_{\ge0}.
\]
For the simple scalar-type modules, we then obtain an explicit isomorphism theorem. Writing
$K(h)=\mathbf{E}_{(c,d)}(u_1,\dots,u_k)$ with $u_i\in\C[h]\setminus\C$, we prove that two such modules are
isomorphic if and only if the parameters $(c,d,u_1,\dots,u_k)$ lie in the same orbit of an explicit
semidirect product group $G=\C^*\rtimes\Z$ acting on these tuples (Theorem~\ref{isomorphismtheorem}).

To keep the presentation focused, we do not pursue here the cohomological viewpoint on
$\sigma^{-1}$-similarity classes (and its relation to $\mathfrak h$-free modules).
\smallskip

The paper is organized as follows. In Section~\ref{sec:prelim} we provide the realization of
$U(\mathfrak h)$-free modules of finite rank with fixed central character, introduce $\sigma^{-1}$-similarity, and
review Cohn's standard form in $\GL_2(\C[h])$. Section~\ref{setS} is devoted to the structure of the set
$\mathcal S$ and to explicit representatives for matrices outside $\mathcal S$. Finally, in
Section~\ref{newfamilysection} we prove the simplicity criterion for scalar-type rank-$2$ modules and
establish the isomorphism theorem using the group $G$.

\section{Preliminaries}\label{sec:prelim}

\subsection{Conventions and notation}
Throughout this paper, we use $\mathbb{Z}$, $\mathbb{C}$, and $\mathbb{C}^*$ to denote the sets of integers, complex numbers, and nonzero complex numbers, respectively. We write $\mathbb{Z}_{\geq k}$ for the set of all integers $i$ such that $i \geq k$.  Unless otherwise stated, all vector spaces, Lie algebras (including $\mathfrak{sl}(2)$), homomorphisms, and tensor products are considered over the field $\mathbb{C}$.

For a Lie algebra $\mathfrak{g}$, we denote its universal enveloping algebra by $U(\mathfrak{g})$.

We fix the basis $\{e, f, h\}$ of the Lie algebra $\mathfrak{sl}(2)$ to be given by the following matrices:
$$
e := \begin{pmatrix} 0 & 1\\ 0 & 0 \end{pmatrix}, \qquad 
f := \begin{pmatrix} 0 & 0\\ 1 & 0 \end{pmatrix}, \qquad 
h := \begin{pmatrix} \frac{1}{2} & 0\\ 0 & -\frac{1}{2} \end{pmatrix}.
$$

The Lie bracket relations among these basis elements are:
\begin{equation*}
\label{commrelations}
[h, e] = e, \qquad [e, f] = 2h, \qquad [h, f] = -f.
\end{equation*}

In this paper we fix the quadratic  Casimir element of $\mathfrak{sl}(2)$ to be $c = (2h + 1)^2 + 4fe$. It is well known that the center   $\mathcal{Z}(\mathfrak{sl}(2))$  of  $U(\mathfrak{sl}(2))$  is $\mathbb{C}[c]$. We will say that an $\mathfrak{sl}(2)$-module $M$ has central character $\gamma \in \C$ if $c$ acts by multiplication by $\gamma$ on $M$.

For a ring $R$, we write $\mathrm{Mat}_n(R)$ for the ring of $n \times n$ matrices with entries in $R$.  We use the same symbol \(h\) for the Cartan element of \(\mathfrak{sl}(2)\) and as the variable of the ring \(\C[h]\). On \(\C[h]^{\oplus n}\), the element \(h \in \mathfrak{sl}(2)\) acts by multiplication by the indeterminate \(h\).
 We denote the group of all invertible $n \times n$ matrices with entries in $\C[h]$ as $\GL_n(\C[h])$. The shift automorphism $\sigma$ of $\mathbb{C}[h]$ will play a crucial role in this paper, namely,   $\sigma:\mathbb{C}[h]\;\to\;\mathbb{C}[h]$ defined by
$$
\sigma\big(p(h)\big)\;:=\;p(h-1),
\qquad\text{for every }p(h)\in\mathbb{C}[h].
$$

We will often write $A'(h)$ for a second matrix in $\Mat_n(\C[h])$ associated to $A(h) \in \Mat_n(\C[h])$; in such cases $A'$ is not the derivative of $A$.

\subsection{The category of $U(\mathfrak{h})$-free modules of finite rank} Define $\mathcal{M}$ to be the full subcategory of
$U(\mathfrak{sl}(2))\text{-mod}$ whose objects are $U(\mathfrak{sl}(2))$-modules $M$ such that
$$\operatorname{Res}^{U(\mathfrak{sl}(2))}_{\;U(\mathfrak{h})} M \;\simeq\; U(\mathfrak{h})^{\oplus n}\;=\;\C[h]^{\oplus n} \qquad\text{for some}\;\; n \in \Z_{\geq 1}. $$

Every object $M$ in $\mathcal{M}$  is  an $\mathfrak{sl}(2)$-module and a $U(\mathfrak h)$-free of rank $n$. For brevity, we will often call such $M$  a \emph{rank-$n$ module}. 

Throughout the paper,  we will identify the underlying space of any module $M$ in $\mathcal M$ with 
$M  =\C[h]^{\oplus n}$ for some $n \in \Z_{\geq 1}$ and assume that the action of $h$ is by multiplication. In explicit terms, every $m \in M$ is  a column vector $ \bigl(g_1(h)\;\;\dots\;\;g_n(h)\bigr)^{\mathsf T}$, with  $g_j(h) \in {\mathbb C} [h] $ and 
$$h \; \cdot \; \bigl(g_1(h)\;\;\dots\;\; g_{n}(h)\bigr)^{\mathsf T} \;=\; \bigl(hg_1(h)\;\;\dots\;\;hg_{n}(h)\bigr)^{\mathsf T}.$$

In this paper, we focus on the full subcategory $\mathcal M(2)\subset \mathcal M$ consisting of those that are $ U(\mathfrak h)$-free of rank $2$. Precisely,
$$
\mathcal M(2):=\Bigl\{\,M\in\mathcal M \ \bigm|\ 
\operatorname{Res}^{\; U(\mathfrak{sl}(2))}_{\;U(\mathfrak h)} M
\;\simeq\; U(\mathfrak h)^{\oplus 2}\Bigr\}.
$$

We next provide a description of  all $\mathfrak{sl}(2)$-modules that are $U(\mathfrak h)$-free of fixed rank $n$ and of fixed central character $\gamma$. The proofs rely on solving matrix equations related to the three defining relations of $\mathfrak{sl}(2)$, and Smith Normal Form technique. The details will appear in  \cite{GNZ}.

\begin{definition} \label{tripledef}
Let ${\bf{a}} = (a_-, a_0, a_+)$ be a triple with $a_-, a_+ \in \mathbb{Z}_{\geq 0}$ and $a_0 \in \mathbb{Z}$ such that $a_- + |a_0| + a_+ = n$.  
We define the diagonal matrix \( P_{({\bf{a}},\alpha)}(x) \in \Mat_n(\mathbb{C}[x]) \) as follows:
\begin{itemize}
    \item[(i)] \( P_{({\bf{a}},\alpha)}(x)_{ii} = 1 \) for \( i = 1, \ldots, a_- \),
    \item[(ii)] If \( a_0 \geq 0 \), then \( P_{({\bf{a}},\alpha)}(x)_{ii} = x - \alpha+1 \) for \( i = a_- + 1, \ldots, a_- + a_0 \),
    \item[(iii)] If \( a_0 < 0 \), then \( P_{({\bf{a}},\alpha)}(x)_{ii} = x + \alpha  \) for \( i = a_- + 1, \ldots, a_- - a_0 \),
    \item[(iv)] \( P_{({\bf{a}},\alpha)}(x)_{ii} = (x - \alpha+1)(x + \alpha) \) for \( i = a_- + |a_0| + 1, \ldots, n \).
\end{itemize}

We also denote $\overline{P}_{({\bf{a}},\alpha)}(x)  \in \Mat_n(\mathbb C[x]) $ so that 
$$\overline{P}_{({\bf{a}},\alpha)}(x)  P_{({\bf{a}},\alpha)}(x)  = -(x - \alpha+1)(x + \alpha) \I_n.$$
\end{definition}

\begin{theorem} \label{U(h)-free-realization}
Let $M = \C[h]^{\oplus n}$ be a module  in  $\mathcal{M}$ with a central character $\gamma = (2\alpha-1)^2$.  Then there exist $K(h)\in\GL_n(\C[h])$ and ${\bf{a}} = (a_-, a_0, a_+)$, such that the $U(\mathfrak{sl}(2))$-action on $M$ is given by:
$$ e\cdot \begin{pmatrix}g_1(h) \\\vdots \\g_n(h)\end{pmatrix} =  \sigma\left(K^{-1}(h)\overline{P}_{({\bf{a}},\alpha)}(h) \begin{pmatrix}g_1(h) \\ \vdots \\g_n(h)\end{pmatrix}\right),\quad f\cdot \begin{pmatrix}g_1(h)  \\\vdots \\g_n(h)\end{pmatrix} = P_{({\bf{a}},\alpha)}(h) K(h) \sigma^{-1} \begin{pmatrix}g_1(h) \\ \vdots \\g_n(h)\end{pmatrix}, $$
where $g_i(h) \in \C[h]$ for all $i \in \{1,\ldots, n\}$. The matrix $P_{({\bf{a}},\alpha)}(h)$ is the Smith Normal Form of the matrix of $f \in \End(M)$ relative to the standard basis of $\C[h]^{\oplus n}$.
\end{theorem}

\begin{theorem} \label{isomorphismthm}
Let $\alpha, \beta \in \C$; let ${\bf a} = (a_-, a_0, a_+)$ and ${\bf b} = (b_-, b_0, b_+)$ be triples as defined in Definition~\ref{tripledef}; and let $K_1(h), K_2(h) \in \GL_n(\C[h])$. Then
$$
M\left( \alpha, {{\bf a}}, K_1(h) \right) \simeq M\left( \beta,{\bf b}, K_2(h) \right)
$$
if and only if the following conditions hold:
\begin{itemize}
    \item either $\alpha = \beta$ and ${\bf a} = {\bf b}$, or $\alpha + \beta = 1$ and ${\bf b} = (a_-, -a_0, a_+)$;
    \item there exists $P(h) \in \GL_n(\C[h])$ such that
    $$
    P(h)P_{({\bf a}, \alpha)}(h)K_2(h) =P_{({\bf a}, \alpha)}(h)K_1(h) P(h+1).
    $$
\end{itemize}
\end{theorem}

\begin{remark}
The theorem above implies that, unless $\gamma = 0$, the modules $M\left(\frac{1}{2}(1 + \gamma^{1/2}), {\bf{a}} , K(h) \right)$ and $M\left(\frac{1}{2}(1 - \gamma^{1/2}), {\bf{a}}, K(h) \right)$ have central character $\gamma$ (here $\gamma^{1/2}$ is a fixed square root of $\gamma$) and in general are nonisomorphic. In view of Theorem \ref{isomorphismthm}, one may fix $\alpha$ with the property $\re(\alpha) \geq 1/2$ and abbreviate $M\left({\bf{a}}, K(h) \right)$ for the module $M\left(\alpha, {\bf{a}}, K(h) \right)$ of central character $\gamma$. In this paper, we prefer to keep $\alpha$ as a part of the parameterization of the module.
\end{remark}

\subsection{$\sigma^{-1}$--similarity}
In this subsection we recall the notions of $\sigma$- and $\sigma^{-1}$-similarity, following \cite{GNZ}. Our emphasis throughout is on $\sigma^{-1}$-similarity.
\begin{definition}[$\sigma$--similarity and $\sigma^{-1}$--similarity]\label{sigmadef}
Let $A, B \in \Mat_n(\mathbb{C}[h])$. We say that $A$ and $B$ are \emph{$\sigma$-similar}, written $A\sim_{\sigma} B$, if there exists $P(h)\in\GL_n(\C[h])$ such that
$$
B \;=\; P(h)^{-1}\,A\,P(h-1) \;=\; P(h)^{-1}A\,\sigma(P(h)).
$$
Similarly, $A$ and $B$ are \emph{$\sigma^{-1}$-similar}, written $A\sim_{\sigma^{-1}} B$, if there exists $P(h)\in\GL_n(\C[h])$ such that
$$
B \;=\; P(h)^{-1}\,A\,P(h+1) \;=\; P(h)^{-1}A\,\sigma^{-1}(P(h)).
$$
\end{definition}

\begin{remark}
The relations of $\sigma$-similarity and $\sigma^{-1}$-similarity on $\Mat_n(\C[h])$ are equivalence relations. The second bullet of Theorem~\ref{isomorphismthm} can be written as
$$
P_{(\mathbf a,\alpha)}(h)\,K_2(h)\ \sim_{\sigma^{-1}}\ P_{(\mathbf a,\alpha)}(h)\,K_1(h),
$$
i.e., these matrices are $\sigma^{-1}$-similar.
\end{remark}

\begin{definition}
For $A(h)\in \Mat_n(\C[h])$, its \emph{$\sigma^{-1}$-orbit} in $\Mat_n(\C[h])$ is
$$
\Orb_{\sigma^{-1}}(A(h))
:= \big\{\, P(h)^{-1} A(h)\, P(h+1) \mid P(h)\in \GL_n(\C[h]) \,\big\}.
$$
Note that if $K(h)\in \GL_n(\C[h])$, then $\Orb_{\sigma^{-1}}(K(h))\subset \GL_n(\C[h])$.
\end{definition}

\begin{remark} \label{rem-def-h1}
Let $\mathcal G=\GL_n(\C[h])$, and let $\Z$ act on $\mathcal G$ via $\sigma^{-1}$, i.e.
$(m \cdot P)(h) := P(h+m)$.
Then the elements of the (nonabelian) cohomology $H^1(\Z,\mathcal G)$ are precisely the
$\sigma^{-1}$-equivalence classes in $\mathcal G$. More precisely, we have a bijection
$\mathcal G/\sim_{\sigma^{-1}} \;\to\; H^1(\Z,\mathcal G)$,
$\Orb_{\sigma^{-1}}(K(h)) \mapsto [c_{K}]$, where $c_K:\Z\to\mathcal G$ is the $1$--cocycle defined by
\[
c_K(0)=\I,\qquad
c_K(m)=\prod_{i=0}^{m-1}K(h+i)\ (m\ge1),\qquad
c_K(-m)=\prod_{i=1}^{m}K(h-i)^{-1}\ (m\ge1).
\]
\end{remark}

\begin{definition}
A rank-$n$ $U(\mathfrak h)$-free module $M(\alpha,\mathbf a,K(h))$ is said to be
\emph{of scalar type} if
\[
P_{(\mathbf a,\alpha)}(h)=p(h)\,\I_n
\]
for some monic polynomial $p(h)\in\C[h]$ that divides $(h-\alpha+1)(h+\alpha)$.
Note that $M(\alpha,\mathbf a,K(h))$  is of scalar type if and only if 
\[
\mathbf a\in\{(n,0,0),(0,n,0),(0,-n,0),(0,0,n)\}.
\]
\end{definition}

\begin{lemma}\label{lemsiginv2}
Let $\alpha\in\C$, $K(h)\in\GL_2(\C[h])$, and ${\bf a}\in\{(2,0,0),(0,2,0),(0,0,2)\}$ (so $M\left(\alpha, \mathbf a, K(h)\right)$ is of scalar type).
Then $M\left(\alpha, \mathbf a, K(h)\right)$ has a rank-$1$ submodule if and only if
\[
K(h)\ \sim_{\sigma^{-1}}\ \begin{pmatrix} a & u(h)\\ 0 & b \end{pmatrix},
\qquad a,b\in\C^*,\ u(h)\in\C[h].
\]
\end{lemma}

\begin{proof}
Write $R=\C[h]$ and $M=R^{\oplus 2}$. In the scalar case there exist $p,q\in R$ with
$p(h)q(h)=-(h-\alpha+1)(h+\alpha)$,  such that,  for all $v\in M$.
\[
e\cdot v=\sigma\big(q(h)\,K(h)^{-1}v\big),\qquad
f\cdot v=p(h)\,K(h)\,\sigma^{-1}(v).
\]

For the ``if part'', observe that if $K(h)=\begin{psmallmatrix}a&u(h)\\ 0&b\end{psmallmatrix}$ with $a,b\in\C^*$, then
$Re_1$ is stable under $e$ and $f$; hence it is a rank-$1$ submodule.

We next prove the ``only if'' part. Assume $L=Rv$ is a rank-$1$ submodule of $M$. Write $v=(v_1,v_2)^{\mathsf T}$ and let
$g(h)=\gcd(v_1,v_2)\in R$; choose a primitive $u\in R^{\oplus 2}$ with $v=g(h)\,u$.
Since $R$ is a PID, there exists $S(h)\in\GL_2(R)$ whose first column is $u$
(i.e. $S(h)e_1=u$). Consider the $\sigma^{-1}$-similar module $M'$ with
\[
K'(h)\ :=\ S(h)^{-1}K(h)S(h+1)\quad\big(\text{so }K\sim_{\sigma^{-1}}K'\big).
\]
Under this similarity, the submodule corresponding to $L$ is
\[
S(h)^{-1}L \;=\; R\cdot S(h)^{-1}\big(g(h)\,u\big)
\;=\; R\cdot g(h)\,e_1.
\]

Write $K'(h)=\begin{psmallmatrix}a(h)&b(h)\\ c(h)&d (h)\end{psmallmatrix}$. 
The $f$–stability of $R\cdot g(h)\,e_1$ implies that for some $\xi(h)\in R$,
\[
\xi(h)\,g(h)\,e_1 \;= \;f\cdot\big(g(h)\,e_1\big)\;=\;p(h)\,K'(h)\,\sigma^{-1}\big(g(h)\,e_1\big)
\;=\;p(h)\,K'(h)\,\big(g(h+1)\,e_1\big)
\]

Comparing the second coordinates gives
\[
p(h)\,c(h)\,g(h+1)\;=\;0.
\]

Hence $c(h)= 0$, and  $K'(h)$ is upper triangular. Because $\det K'=\det K\in\C^\times$, both diagonal
entries $a(h),d(h)$ are units of $R$.
Therefore
\[
K'(h)=\begin{pmatrix} a & u(h)\\[2pt] 0 & b \end{pmatrix}
\quad\text{for some }u(h)\in R,
\]
for some $a,b\in\C^*$ as required.
\end{proof}

\begin{definition}\label{1st-definitionofS}
We define $\mathcal S$ to be the complement in $\GL_2(\C[h])$ of the union of the $\sigma^{-1}$-similarity equivalence classes of constant diagonal matrices:
$$
\mathcal S := \GL_2(\C[h]) \setminus \bigcup_{a,b\in\C^*} \Orb_{\sigma^{-1}}\bigl(\diag(a,b)\bigr).
$$
\end{definition}

A detailed understanding of $\mathcal S$ is essential for our main theorem.

\subsection{Standard Form for Matrices in $\GL_2(\C[h])$} \label{standardformdef} 
Throughout, we use the notation
$$
E(u(h)) := \begin{pmatrix} u(h) & 1 \\ -1 & 0 \end{pmatrix}\quad (u(h)\in\C[h]).
$$

\begin{lemma}[\cite{Co}, Theorem 2.2]\label{shortenlem}
For $ u(h), v(h), w(h) \in \mathbb{C}[h] $ and $ \beta_1, \beta_2 \in \mathbb{C}^* $, the following identities hold:
\begin{enumerate}
    \item[(i)] $ E(u(h))E(0)E(v(h)) = -E(u(h)+v(h)) $,
    \item[(ii)] $ E(\beta_1)E\left(\beta_1^{-1}\right)E(\beta_1) = -\diag\left(\beta_1,\beta_1^{-1}\right)$,
    \item[(iii)] $ E(u(h)) \diag(\beta_1, \beta_2) = \diag(\beta_2, \beta_1)E\left(\tfrac{\beta_1}{\beta_2}u(h)\right) $,
    \item[(iv)] $ E(u(h))E(v(h))^{-1} = E(u(h)-v(h))E(0)^{-1} = -E(u(h)-v(h))E(0) $,
    \item[(v)] $ E(u(h))E(v(h))^{-1}E(w(h)) = E(u(h)-v(h)+w(h)) $,
    \item[(vi)] $ E(u(h))E(\beta_1)E(v(h)) = E\left(u(h)-\beta_1^{-1}\right) \diag\left(\beta_1, \beta_1^{-1}\right) E\left(v(h)-\beta_1^{-1}\right)$.
\end{enumerate}
\end{lemma}

\begin{remark}\label{inverseremark}
By Lemma~\ref{shortenlem}(iv) with $u(h)=0$, for every $v(h)\in\C[h]$,
$$
E(v(h))^{-1}=E(0)\,E(-v(h))\,E(0).
$$
\end{remark}
\begin{theorem}[\cite{Co}, Theorem~2.2]\label{GL2form}
Every matrix $A\in\GL_2(\C[h])$ admits a factorization of the form
$$
A=\diag(\beta_1,\beta_2)\qquad \text{or}\qquad A=\diag(\beta_1,\beta_2)\,\prod_{i=1}^k E\big(u_i(h)\big),
$$
for some $\beta_1,\beta_2\in\C^*$, $k\geq1$, $u_i(h)\in\C[h]$.
\end{theorem}

For brevity, set $  \mathbf{E}_{(\beta_1, \beta_2)}( u_1,\dots,u_k):= \diag(\beta_1,\beta_2)\,\prod_{i=1}^k E(u_i)$. By Lemma~\ref{shortenlem}, the factorization of $A$ from Theorem~\ref{GL2form} can be shortened in finitely many steps to a reduced expression $ A =\mathbf{E}_{(\omega_1, \omega_2)}(v_1,\dots,u_\ell)$, where $\omega_1,\omega_2\in\C^*$, $v_1(h),v_\ell(h)\in\C[h]$, and $v_i(h)\in\C[h]\setminus\C$ for $1<i<\ell$.  We call this the \emph{standard form} of $A$.

\begin{remark}\label{length-2remark}
If $\mathbf{E}_{(\beta_1, \beta_2)}(u_1,u_2)$ is the standard form of $A \in \GL_2(\C[h])$, then $(u_1,u_2) \neq (0,0)$. For $m \geq n$, set
$$\mathop{\overleftarrow{\prod}}_{i=n}^m E\bigl(u_i(h)\bigr)  :=   E\bigl(u_m(h)\bigr)\,E\bigl(u_{m-1}(h)\bigr)\cdots E\bigl(u_n(h)\bigr).  $$
\end{remark}

\begin{theorem}[\cite{Co}, Theorem~7.1]\label{standardform}
Every $A\in\GL_2(\C[h])$ admits a unique standard form.
\end{theorem}

\begin{definition}
Let $K(h)\in \GL_2(\C[h])$. We define the \emph{length} of $K(h)$ (denoted by $\ell(K(h))$) to be the number of factors of the form $E\bigl(u_i(h)\bigr)$ appearing in the standard form of $K(h)$. In particular, if $K(h)=\diag(\beta_1,\beta_2)$ with $\beta_1,\beta_2\in\C^*$, then $\ell(K(h))=0$.
\end{definition}

\subsection{Standard form via a Euclidean Algorithm on the first row}\label{subsec:std-form}

We assume $\deg 0:=-\infty$.

\begin{proposition}\label{prop:standard-form}
Let 
\[
K(h)=\begin{pmatrix}k_{11}&k_{12}\\ k_{21}&k_{22}\end{pmatrix}\in\GL_2(\C[h]).
\]
Then there exist $a,b\in\C^*$ and $u(h), q_{i}(h) \in\C[h]$ with $q_i\notin\C^*$ for every $i>1$, such that
\begin{equation} \label{eq-lq-form}
K(h)\ =\ \begin{pmatrix}a&0\\ u(h)&b\end{pmatrix}\,E\big(q_T\big)\cdots E\big(q_1\big).
\end{equation}
Moreover,  $q_i = (-1)^{i-1}p_i$, for $i=1,...,T$, where $p_1,...,p_T$ are the consecutive quotients obtained via the Euclidean Algorithm applied on $(k_{11}, k_{12})$. In particular, $q_1=0$ if $\deg k_{11}<\deg k_{12}$, and  $T=0$ if $k_{12}=0$.
\end{proposition}

\begin{proof}
Set $K^{(1)}:=K$ and write
\[
K^{(t)}=\begin{pmatrix}k^{(t)}_{11}&k^{(t)}_{12}\\ k^{(t)}_{21}&k^{(t)}_{22}\end{pmatrix}\qquad (t\ge1).
\]

While $k^{(t)}_{12}\neq 0$, perform Euclidean division
\[
k^{(t)}_{11}=q_t\,k^{(t)}_{12}+r_t,\qquad \deg r_t<\deg k^{(t)}_{12},
\]
and set
\[
K^{(t+1)}:=
\begin{pmatrix}
k^{(t)}_{12} & -\,k^{(t)}_{11}+q_t\,k^{(t)}_{12}\\[2pt]
k^{(t)}_{22} & -\,k^{(t)}_{21}+q_t\,k^{(t)}_{22}
\end{pmatrix}.
\]

Then
\[
K^{(t+1)}\,E(q_t)
=\begin{pmatrix}
q_t k^{(t)}_{12}+r_t & k^{(t)}_{12}\\
k^{(t)}_{21} & k^{(t)}_{22}
\end{pmatrix}
=K^{(t)}.
\]

Moreover, $k^{(t+1)}_{11}=k^{(t)}_{12}$ and $k^{(t+1)}_{12}=-r_t$, hence
$\deg k^{(t+1)}_{12}<\deg k^{(t)}_{12}$, so the process terminates after finitely many steps at some $T$ with
$k^{(T+1)}_{12}=0$. Thus
\[
K=K^{(T+1)}\,E(q_T)\cdots E(q_1).
\]

Write $K^{(T+1)}=\begin{psmallmatrix}A&0\\ U&B\end{psmallmatrix}$ with $A,B,U\in\C[h]$. Since 
$\det K=\det K^{(T+1)}=AB\in\C^*$, we have $A,B\in\C^*$. Setting $a:=A$, $b:=B$, $u:=U$ gives the factorization.

Finally, if $\deg k_{11}<\deg k_{12}$ then $q_1=0$ by definition; if $k_{12}=0$ then $T=0$ and 
$K=\begin{psmallmatrix}k_{11}&0\\ k_{21}&k_{22}\end{psmallmatrix}$ with $k_{11},k_{22}\in\C^*$.
For $i>1$, we have $\deg k^{(i)}_{11}>\deg k^{(i)}_{12}$ (since for $t \geq 2$, $k^{(t)}_{11}=k^{(t-1)}_{12}$ and $\deg k^{(t)}_{12}<\deg k^{(t-1)}_{12}$), hence $\deg q_i\ge1$, equivalently, $q_i\notin\C^*$.
\end{proof}

\begin{proposition}
The form of $K(h)$ in \eqref{eq-lq-form} is unique. Namely,
assume
\[
K(h)\;=\;\begin{pmatrix}a&0\\ u&b\end{pmatrix}E(q_T)\cdots E(q_1)
\;=\;\begin{pmatrix}\tilde a&0\\ \tilde u&\tilde b\end{pmatrix}E(\tilde q_{\tilde T})\cdots E(\tilde q_1),
\]
where $a,b,\tilde a,\tilde b\in\C^*$, $u,\tilde u\in\C[h]$, and $q_i,\tilde q_j\in\C[h]$ with
$q_i\notin\C^*$ for every $i>1$ and $\tilde q_j\notin\C^*$ for every $j>1$. Then $a=\tilde a$, $b=\tilde b$, $u= \tilde u$, $T= \tilde T$, and $q_i= \tilde q_i$ for all $1\leq i \leq T$.
\end{proposition}

\begin{proof}
If $\tilde T = 0$, then the identity $\begin{pmatrix}a&0\\ u&b\end{pmatrix}E(q_T)\cdots E(q_1) = \begin{pmatrix}\tilde a&0\\ \tilde u&\tilde b\end{pmatrix}$ implies
$$\I_2 = \tfrac{1}{\tilde a \tilde b}\begin{pmatrix}\tilde b a& 0 \\-\tilde u a + \tilde a u & \tilde a b \end{pmatrix}E(q_T)\cdots E(q_1)= \tfrac{1}{\tilde a \tilde b} \diag(-\tilde b a,-\tilde a b) E(0) E(\tfrac{-\tilde u a + \tilde a u}{\tilde a b})E(q_T)\cdots E(q_1).$$

Hence, if $T  \geq 1$, the identity matrix $\I_2$ would admit a standard form involving a factor $E(w)$ with $w \in \C[h]$, which is a contradiction. Thus $T = 0$, and it follows that $a = \tilde a$, $b = \tilde b$, and $u = \tilde u$.

It remains to consider the case where $T \geq 1$ and $\tilde T \geq 1$. From 
\[\begin{pmatrix}a&0\\ u&b\end{pmatrix}E(q_T)\cdots E(q_1)
\;=\;\begin{pmatrix}\tilde a&0\\ \tilde u&\tilde b\end{pmatrix}E(\tilde q_{\tilde T})\cdots E(\tilde q_1),
\]
it follows that
\begin{align*}
\I_2 &= (-1)^{\tilde T-1} \tfrac{1}{\tilde a \tilde b}\begin{pmatrix}\tilde b a& 0 \\-\tilde u a + \tilde a u & \tilde a b \end{pmatrix}E(q_T)\cdots E(q_1) E(0) E(-\tilde q_1)\cdots E(-\tilde q_{\tilde T}) E(0)\\
&=(-1)^{\tilde T} \tfrac{1}{\tilde a \tilde b}\begin{pmatrix}\tilde b a& 0 \\-\tilde u a + \tilde a u & \tilde a b \end{pmatrix}E(q_T)\cdots E(q_2)E(q_1-\tilde q_1)E(-\tilde q_2)\cdots E(-\tilde q_{\tilde T}) E(0).
\end{align*}

If $q_1 \neq \tilde q_1$, then either $q_1 - \tilde q_1 \in \C[h]\setminus \C$ or $q_1 - \tilde q_1 \in \C^*$. If $\tilde T \geq 2$, then $E(-\tilde q_{\tilde T})$ appears in the standard form of $\I_2$, which contradicts Theorem~\ref{standardform}. If $T \geq 2$, then either $E(q_1-\tilde q_1)$ appears in the standard form of $\I_2$ (i.e. when $q_1 - \tilde q_1 \in \C[h]\setminus \C$), or, in the case $q_1 - \tilde q_1 \in \C^*$, we may apply Lemma~\ref{shortenlem} parts (iii) and (vi) to bring the product into standard form, yielding a contradiction. If $T= \tilde T=1$, then  
$$\begin{pmatrix} aq_1 & a \\ uq_1-b & u \end{pmatrix}= \begin{pmatrix}\tilde a\tilde q_1 & \tilde a \\ \tilde u\tilde q_1-\tilde b & \tilde u \end{pmatrix} \quad \text{and hence} \quad q_1 =  \tilde q_1.$$

 Therefore, $q_1 = \tilde q_1$. Assume $T > \tilde T$ and apply the above argument inductively, we obtain $q_i = \tilde q_i$ for all $1\leq i \leq \min\{ T, \tilde T\} = \tilde T$. Hence
 $$ \I_2 = \tfrac{1}{\tilde a \tilde b}\begin{pmatrix}\tilde b a& 0 \\-\tilde u a + \tilde a u & \tilde a b \end{pmatrix}E(q_T)\cdots E(q_{T- \tilde T})= \tfrac{1}{\tilde a \tilde b} \diag(-\tilde b a,-\tilde a b) E(0) E(\tfrac{-\tilde u a + \tilde a u}{\tilde a b})E(q_T)\cdots E(q_{T- \tilde T}).$$
 
 Using the same argument as in the case $\tilde T = 0$, we obtain a contradiction. Thus $T = \tilde T$ and $q_i = \tilde q_i$ for all $1 \le i \le T$. Consequently,
$$\begin{pmatrix}a&0\\ u&b\end{pmatrix} = \begin{pmatrix}\tilde a&0\\ \tilde u&\tilde b\end{pmatrix},$$
which proves the lemma.
\end{proof}

We will call the decomposition \eqref{eq-lq-form} the \emph{LQ form of $K(h)$}.

\begin{remark}[Transition from LQ form to Standard Form]
By Proposition~\ref{prop:standard-form},
\[
K(h)\;=\;\begin{pmatrix}a&0\\ u(h)&b\end{pmatrix} \, E\big(q_T\big)\cdots E\big(q_1\big),
\qquad a,b\in\C^*,\ u\in\C[h],
\]
with $q_i \notin \C^*$ for $i>1$.  Rewrite the lower–triangular front factor as
\[
\begin{pmatrix}a&0\\ u(h)&b\end{pmatrix}
=\diag(-a,-b)\,E(0)\,E\Big(\tfrac{u(h)}{b}\Big).
\]

If $T=0$ and $u \neq 0$, the above is the normal form of $K(h) = \begin{pmatrix}a&0\\ u(h)&b\end{pmatrix}$. In this case $\ell(K(h))= 2$.

If $u=0$, then the standard form of $K(h)$ is $ \mathbf{E}_{(a, b)}( q_T,q_{T-1},\dots,q_1)$, and hence $\ell(K(h)) = T$.

If $T\geq 1$ and  $\frac{u(h)}{b}\notin\C$, then, since the product of $E(q_i)$ has no interior constants $q_i$,  the resulting product is already in the desired standard form.  In this case $\ell (K(h)) = T+2$.

If $T\geq 1$ and $\frac{u(h)}{b}=\beta \in\C^*$, by  Lemma~\ref{shortenlem} we have
$$\begin{aligned}
E(0)\,E\left(\beta\right)E\big(q_T(h)\big) &= E\left(-\beta^{-1} \right) \diag\left(\beta, \beta^{-1}\right) E\big(q_T(h)-\beta^{-1}\big) \\
&= \diag\left(\beta^{-1}, \beta\right) E(-\beta)E\big(q_T(h)-\beta^{-1}\big).
\end{aligned}$$

This combined with $E\big(q_{T-1}\big)\cdots E\big(q_1\big)$ equals the normal form of $K(h)$, and in this case $\ell (K(h)) = T+1$.
\end{remark}

\begin{example} \label{ex-lq-form}
\begin{itemize}
\item[(i)] The following
$$
K_1(h) = \begin{pmatrix}-h^4-2h & -h^3-1 \\ -h^6-h^3+1 & h^5 \end{pmatrix} = \begin{pmatrix}1 & 0 \\ h^2 & 1 \end{pmatrix} E(h)E(-h^2)E(h)$$ is the LQ form of $K_1(h)$, because after applying the Euclidean algorithm to the pair $(-h^4-2h,-h^3-1)$, we produce the three quotients $h, h^2, h$. The standard form of $K_1(h)$ is  
$$K_1(h) =\diag(-1,-1)E(0)E(h^2) E(h)E(-h^2)E(h).$$
\item[(ii)] The following
\[
K_2(h)=\begin{pmatrix}-2h & -2h^3-2h-2\\ -h+3 & -h^3+3h^2-h+2\end{pmatrix} =\;\begin{pmatrix}2&0\\ 1&3\end{pmatrix}E(h)E(-h^2-1)E(0)
\]
is the LQ form of $K_2(h)$, because after applying the Euclidean algorithm to the pair $(-2h,-2h^3-2h-2)$, we produce the three quotients $0, h^2+1,h$.  The standard form of $K_2(h)$ is
$$K_2(h)= \diag(-6,-1)E\left(-\tfrac{1}{3} \right)E(h-3)E(-h^2-1)E(0).$$

\end{itemize}
\end{example}

\section{The Set $\mathcal S$} \label{setS}
This section provides an alternative characterization of $\mathcal S$ from Definition~\ref{1st-definitionofS}, a key parameter in the classification of simple rank-2 scalar-type modules (Corollary \ref{meaningofS}). In particular, we give an explicit formula for matrices in $\mathcal S$, up to $\sigma^{-1}$-similarity. We begin with a standard lemma that will be used in several places.

\begin{lemma}\label{surj-lem-ab}
Let $a,b\in\C^*$. Define
\[
T_{a,b}:\C[h]\longrightarrow\C[h],\qquad
T_{a,b}\bigl(w(h)\bigr)=a\,w(h+1)-b\,w(h).
\]
Then $T_{a,b}$ is a surjective linear map.
\end{lemma}

\begin{proof}  Since $T_{a,b} = aT_{1,b/a}$, it is enough to show that $S_c := T_{1,c}$ is surjective for every
$c\in\C^*$. Consider the basis of  $\C[h]$ consisting of the  binomial polynomials
\[
e_k(h):=\binom{h}{k}\qquad(k\ge 0).
\]

Using the identity $\binom{h+1}{k}=\binom{h}{k}+\binom{h}{k-1}$, we
obtain
\[
S_c(e_k)(h)
=
e_k(h+1) - c\,e_k(h)
=
\bigl(1-c\bigr)e_k(h) + e_{k-1}(h),
\]
with the convention $e_{-1}:=0$.

\smallskip\noindent
\textbf{Case 1: $c\neq 1$.}
In the basis $\{e_0,e_1,e_2,\dots\}$ the operator $S_c$ has an
upper-triangular matrix with diagonal entries $1-c\neq 0$. Hence $S_c$
is invertible, in particular surjective.

\smallskip\noindent
\textbf{Case 2: $c=1$.}
Then $S_1(e_k)=e_{k-1}$ for all $k\ge 0$. Thus $S_1$ is the (right)
shift operator on the  binomial basis:
\[
S_1(e_0)=0,\quad S_1(e_1)=e_0,\quad S_1(e_2)=e_1,\ \dots
\]

This easily implies the surjectivity of $S_1$.
\end{proof}

\begin{lemma} \label{lemsiginv}
Let $a,b\in\C^*$ and $u(h)\in\C[h]$. Then 
\[
\begin{pmatrix} a & u(h) \\0 & b \end{pmatrix}
\in \Orb_{\sigma^{-1}}\big(\diag(a,b)\big).
\]
\end{lemma}

\begin{proof}
For $v(h)\in\C[h]$, set $P(h):=
\begin{pmatrix}1&v(h)\\0&1\end{pmatrix}\in\GL_2(\C[h])$.
Then
\[
P(h)^{-1}
\begin{pmatrix}a & u(h)\\0 & b\end{pmatrix}
P(h+1)
=
\begin{pmatrix}
a & a\,v(h+1)-b\,v(h)+u(h)\\
0 & b
\end{pmatrix}.
\]

We use Lemma \ref{surj-lem-ab}, choose $v(h)$ such that $T_{a,b}(v) = -u$. This completes the proof. 
\end{proof}

\begin{remark} \label{constantmatrices}
Any matrix in $\GL_2(\C)$ (i.e.\ with constant entries) is $\sigma^{-1}$-similar to a diagonal matrix.
\end{remark}

\begin{corollary} \label{meaningofS}
Let $\alpha\in\C$, $K(h)\in\GL_2(\C[h])$, and 
${\bf a}\in\{(2,0,0),(0,2,0),(0,0,2)\}$.  
Then $M\left(\alpha, \mathbf a, K(h)\right)$ has no rank-$1$ submodules 
precisely when $K(h)\in\mathcal S$.
\end{corollary}

\begin{proof}
This is an immediate consequence of Lemma~\ref{lemsiginv2} together with Lemma~\ref{lemsiginv}.
\end{proof}

\begin{corollary}
Let $a,b\in\C^*$, $\alpha\in\C$, and $u(h)\in\C[h]$. Then:
\begin{enumerate}
\item[(i)] 
$M\left(\alpha, (2,0,0), 
\begin{pmatrix} a & u(h) \\ 0 & b \end{pmatrix}\right)
\simeq 
M(\alpha, (1,0,0), a)
\;\oplus\;
M(\alpha, (1,0,0), b),$

\item [(ii)]
$M\left(\alpha, (0,2,0), 
\begin{pmatrix} a & u(h) \\ 0 & b \end{pmatrix}\right)
\simeq 
M(\alpha, (0,1,0), a)
\;\oplus\;
M(\alpha, (0,1,0), b),$

\item [(iii)]
$M\left(\alpha, (0,0,2), 
\begin{pmatrix} a & u(h) \\ 0 & b \end{pmatrix}\right)
\simeq 
M(\alpha, (0,0,1), a)
\;\oplus\;
M(\alpha, (0,0,1), b).$
\end{enumerate}
\end{corollary}

\begin{proof}
We prove (i); parts (ii)–(iii) are identical.  
By Lemma~\ref{lemsiginv},
\[
\begin{pmatrix} a & u(h) \\ 0 & b \end{pmatrix}
\sim_{\sigma^{-1}} \diag(a,b).
\]

Hence, by Theorem~\ref{isomorphismthm},
\[
M\left(\alpha, (2,0,0), 
\begin{pmatrix} a & u(h) \\ 0 & b \end{pmatrix}\right)
\simeq 
M\left(\alpha, (2,0,0), \diag(a,b)\right).
\]

Finally,
\[
M\left(\alpha, (2,0,0), \diag(a,b)\right)
\simeq 
M(\alpha, (1,0,0), a)
\;\oplus\;
M(\alpha, (1,0,0), b),
\]
which proves (i).
\end{proof}

\begin{lemma} \label{lengthlemma}
For $k \in \Z_{\geq 1}$, let  $u_1,\dots,u_k\in\C[h]\setminus\C$. Let 
$$P(h)=\mathbf E_{(1,1)}(v_1,\dots,v_p)$$ 
be in standard form, where $v_i\in\C[h]$ for all $i$, $v_i\notin\C$ for $1<i<p$, and $p \in \Z_{\geq1}$. Then the following hold.
\begin{enumerate}
\item[(i)] If $v_1\in\C^*$, then
$$\ell\Big(P(h)^{-1}\,\mathbf{E}_{(1,1)}(u_1,\dots,u_k)\,P(h+1)\Big)= \begin{cases} k+2 &\text{if}\quad p=1,\\
k+2p &\text{if}\quad p>1,\, v_p \in \C[h]\setminus \C,\\
k+2p-1 &\text{if}\quad p>1,\, v_p \in \C^*,\\
k+2p-2 &\text{if}\quad p>1,\, v_p =0. \end{cases}$$
\item[(ii)] If $v_1\in\C[h] \setminus \C$, and $u_1 - v_1 \in \C[h] \setminus \C$, then
$$\ell\Big(P(h)^{-1}\,\mathbf{E}_{(1,1)}(u_1,\dots,u_k)\,P(h+1)\Big)= \begin{cases} k+2p &\text{if}\quad p \geq 1,\, v_p \in \C[h]\setminus \C,\\
k+2p-1 &\text{if}\quad p>1,\, v_p \in \C^*,\\
k+2p-2 &\text{if}\quad p>1,\, v_p =0. \end{cases}$$
\item[(iii)] If $v_1\in\C[h] \setminus \C$ and $u_1 - v_1 \in \C^*$, then
$$\ell\Big(P(h)^{-1}\,\mathbf{E}_{(1,1)}(u_1,\dots,u_k)\,P(h+1)\Big)= \begin{cases} k+2p-1 &\text{if}\quad p \geq 1,\, v_p \in \C[h]\setminus \C,\\
k+1, &\text{if}\quad p=2,\, v_2=(v_1 - u_1)^{-1},\\
k+2, &\text{if}\quad p=2,\, v_2 \in \C, v_2\neq (v_1 - u_1)^{-1},\\
k+2p-2 &\text{if}\quad p>2,\, v_p \in \C^*,\\
k+2p-3 &\text{if}\quad p>2,\, v_p =0. \end{cases}$$
\end{enumerate}
\end{lemma}

\begin{proof}
We prove (i); parts (ii)–(iii) are similar. If $p=1$, then
$$P(h)^{-1}\,\mathbf{E}_{(1,1)}(u_1,\dots,u_k)\,P(h+1)
=-\,E(0)\,E(u_1-v_1)\,E(u_2)\cdots E(u_k)\,E(v_1(h+1)).$$

Hence, when $p=1$,
 $$\ell\Big(P(h)^{-1}\,\mathbf{E}_{(1,1)}(u_1,\dots,u_k)\,P(h+1)\Big) =k+2.$$ 
 
Now assume $p >1$. A direct computation gives
\begin{multline*}
P(h)^{-1}\,\mathbf{E}_{(1,1)}(u_1,\dots,u_k)\,P(h+1)=\\
(-1)^p\,E(0)\left(\mathop{\overleftarrow{\prod}}_{j=2}^{\,p} E(-v_j) \right)\,E(u_1-v_1)\,E(u_2)\cdots E(u_k)\,\left(\prod_{m=1}^{\,p} E(v_m(h+1)) \right).
\end{multline*}

If $v_p \in \C[h]\setminus \C$, then $\ell\Big(P(h)^{-1}\,\mathbf{E}_{(1,1)}(u_1,\dots,u_k)\,P(h+1)\Big) =k+2p$. If $v_p \in \C^*$, then by Lemma \ref{shortenlem} parts (iii) and (vi), we deduce that
$$\ell\Big(P(h)^{-1}\,\mathbf{E}_{(1,1)}(u_1,\dots,u_k)\,P(h+1)\Big) =k+2p-1.$$

If $v_p =0$, then
\begin{multline*}
P(h)^{-1}\,\mathbf{E}_{(1,1)}(u_1,\dots,u_k)\,P(h+1)=\\
(-1)^{p-1}\,\left(\mathop{\overleftarrow{\prod}}_{j=2}^{\,p-1} E(-v_j) \right)\,E(u_1-v_1)\,E(u_2)\cdots E(u_k)\,\left(\prod_{m=1}^{\,p-1} E(v_m(h+1)) \right)E(0),
\end{multline*}
which implies 
$$\ell\Big(P(h)^{-1}\,\mathbf{E}_{(1,1)}(u_1,\dots,u_k)\,P(h+1)\Big) =k+2p-2.$$

This completes the proof of part (i) of the lemma.
\end{proof}

For  $u_1,\dots,u_k\in\C[h]\setminus\C$, the length of $\mathbf{E}_{(a,b)}(u_1,\dots,u_k)$ cannot be reduced via $\sigma^{-1}$-similarity; this will be proved in the next proposition.

\begin{proposition} \label{sigmashorten}
Let $a,b\in\C^*$ and $u_1,\dots,u_k\in\C[h]\setminus\C$ with $k \geq 1$. 
Then for every $P(h)\in\GL_2(\C[h])$,
$$
  \ell\Big(P(h)^{-1}\,\mathbf{E}_{(a,b)}(u_1,\dots,u_k)\,P(h+1)\Big)
  \;\geq\;
  \ell\Big(\mathbf{E}_{(a,b)}(u_1,\dots,u_k)\Big).
$$
\end{proposition}

\begin{proof}
By Lemma~\ref{shortenlem}(iii), the diagonal factor in the standard form of a matrix in $\GL_2(\C[h])$ does not affect the length under $\sigma^{-1}$-similarity. Hence, for computational convenience, it suffices to prove the proposition in the case $a=b=1$. Let $P(h)\in \GL_2(\C[h])$ of standard form
$$
P(h)=\mathbf{E}_{(\beta_1,\beta_2)}(v_1,\dots,v_p),\qquad \beta_1,\beta_2\in\C^*,
$$
with $v_1,v_p\in\C[h]$ and $v_i\in \C[h]\setminus\C$ for $1<i<p$. Since conjugation by a constant diagonal matrix preserves length (Lemma~\ref{shortenlem}(iii)), we also assume $\beta_1= \beta_2 =1$. By Lemma~\ref{lengthlemma}, it remains to consider the cases $v_1=0$ and $v_1=u_1$.
If $v_1= 0$, then
  $$E(v_1)^{-1}\,\mathbf{E}_{(1,1)}(u_1,\dots,u_k)\,E\bigl(v_1(h+1)\bigr)
=-\,E(0)\,E(u_1)\,E(u_2)\cdots E(u_k)\,E(0),$$
and
  \begin{multline*}
  E(v_2)^{-1}E(v_1)^{-1}\,\mathbf{E}_{(1,1)}(u_1,\dots,u_k)\,E\bigl(v_1(h+1)\bigr)\,E\bigl(v_2(h+1)\bigr)\\
=-\,E(0)\,E(-v_2)\,E(u_1)\,E(u_2)\cdots E(u_{k-1})\,E(u_k + v_2(h+1)).
  \end{multline*}
  
If $v_1=0$ and $p\le2$, a direct computation shows that
$$
\ell\Big(P(h)^{-1}\,\mathbf E_{(1,1)}(u_1,\dots,u_k)\,P(h+1)\Big)\geq k+1.
$$

If $p \geq 3$, then
\begin{multline*}
P(h)^{-1}\,\mathbf{E}_{(1,1)}(u_1,\dots,u_k)\,P(h+1)=\\
(-1)^{p+1}\,E(0)\left(\mathop{\overleftarrow{\prod}}_{j=2}^{\,p} E(-v_j) \right)\,E(u_1)\cdots E(u_{k-1})\,E(u_k + v_2(h+1))\left(\prod_{m=3}^{\,p} E(v_m(h+1)) \right).
\end{multline*}

To reach the minimal possible length in the case $p\geq 3$, the interior factors of $P(h)$ must cancel successively with the tail of $\mathbf{E}_{(1,1)}(u_1,\dots,u_k)$; that is, 
  $$v_1(h)=0,\quad v_2(h) = -u_k(h-1),\quad v_3(h) = -u_{k-1}(h-1),\quad \dots,v_p(h)=0.$$ 
  
In explicit terms, $v_1(h) = v_p(h) =0$, and for $2\leq n \leq p-1$, writing $n= ks+r+2$ with $s \geq 0$ and $r \in \{0,\dots, k-1\}$, we set
$$v_{ks+r+2}(h) := -u_{k-r}(h-(s+1)). $$

The resulting product has length $k$.

If $v_1=u_1$, then  
  $$E(v_1)^{-1}\,\mathbf{E}_{(1,1)}(u_1,\dots,u_k)\,E\bigl(v_1(h+1)\bigr)=E(u_2)\cdots E(u_k)E(u_1(h+1)).$$
  
To obtain the minimal length, repeat this process at each stage (i.e., take $v_2(h)=u_2(h)$,$\dots$). More precisely, for $s \geq 0$ and $j \in \{1,\dots,k\}$, set $v_{ks+j}(h):= u_j(h+s)$. The resulting product has length $k$. 
\end{proof}

\begin{corollary}\label{examplecoro}
Let $a,b\in\C^*$, $k\in \Z_{\geq 1}$, and let $u_1,\dots,u_k\in\C[h]\setminus\C$. Then
$$
\mathbf{E}_{(a,b)}(u_1,\dots,u_k)\in\mathcal S.
$$
\end{corollary}
\begin{proof}
Let $K(h)=\mathbf{E}_{(a,b)}(u_1,\dots,u_k)$. By Proposition~\ref{sigmashorten}, every matrix $\sigma^{-1}$-similar to $K(h)$ has length at least $k$, so $K(h)$ cannot be $\sigma^{-1}$-similar to a diagonal matrix (which has length $0$). Hence $K(h)\in\mathcal S$.
\end{proof}

\begin{lemma} \label{lemmaNS-S}
Let $a, b, \beta \in \mathbb{C}^*$, $\theta \in \mathbb{C}$, and $u(h) \in \mathbb{C}[h] \setminus \mathbb{C}$. Then
$$(i) \; \mathbf{E}_{(a,b)}(\theta) \notin \mathcal{S},\qquad (ii)\; \mathbf{E}_{(a,b)}(0,u(h)) \notin \mathcal{S}, \qquad (iii)\; \mathbf{E}_{(a,b)}(\beta,u(h)) \in \mathcal{S}.$$
\end{lemma}

\begin{proof}
Since $\mathbf{E}_{(a,b)}(\theta)\in\GL_2(\C)$, Lemma~\ref{lemsiginv} and Remark~\ref{constantmatrices} yield (i). To prove (ii), we apply Lemma \ref{surj-lem-ab} and choose $w(h) \in \C[h] \setminus \C$ such that $u(h) = -\frac{a}{b}w(h) + w(h+1)$. Then
\begin{align*}
\mathbf{E}_{(a,b)}(0,u(h))
&= \diag(a,b) E\left(0\right)E\left(-\tfrac{a}{b}w(h) + w(h+1)\right) \\
&=  \diag(-a,-b) E(0)E\left(-\tfrac{a}{b}w(h)\right)E(0)E(w(h+1)) \\
&= E(0)E(-w(h))E(0) \diag(-b,-a) E(w(h+1)) \\
&= E(w(h))^{-1}  \diag(-b,-a) E(w(h+1)) \sim_{\sigma^{-1}}  \diag(-b,-a).
\end{align*}

To prove (iii), we first define $P(h) := E(0)E\left( -u(h-1) + \beta^{-1} \right)E(0)$. Then
$$P^{-1}(h) \mathbf{E}_{(a,b)}(\beta,u(h)) P(h+1) = \diag\left(b\beta^{-1},a\beta\right) E\left(u(h-1)\tfrac{a\beta^2}{b} - \tfrac{a\beta}{b} - \beta \right).$$

Then by Corollary \ref{examplecoro},  (iii) follows.
\end{proof}

\begin{theorem}\label{mainstructureS}
Let $K(h)\in\GL_2(\C[h])$. Then $K(h)\in\mathcal S$ if and only if $K(h)$ is
$\sigma^{-1}$-similar to a matrix of the form
$$
\mathbf{E}_{(\beta_1,\beta_2)}(v_1,\dots,v_m),\quad \text{where} \qquad \beta_1,\beta_2\in\C^*,\; m\geq1,\; v_i(h)\in\C[h]\setminus\C.
$$
\end{theorem}

\begin{proof}
The “if” part is immediate from Corollary~\ref{examplecoro}. For the “only if” part, assume $K(h)\in\mathcal S$.  We first observe that for any $\beta\in\C$, $a,b\in\C^*$, and $k\in\Z_{\geq 1}$, we have
$$
\diag(a,b)\left( \prod_{i=1}^k E(u_i(h))\right) E(\beta) \sim_{\sigma^{-1}}  \diag(b,a)E\left( \tfrac{a\beta}{b} \right)\left( \prod_{i=1}^k E(u_i(h))\right).
$$

If $u_1(h) =\eta \in \C$ and $k \geq 2$ then
\begin{align*}
&\quad\diag(a,b)E(\eta)\left( \prod_{i=2}^k E(u_i(h))\right) E(\beta)\\
&\sim_{\sigma^{-1}}\diag(b,a)E\left( \tfrac{a\beta}{b} \right) E(\eta) \left( \prod_{i=2}^k E(u_i(h))\right) \\
&= \begin{cases}
\diag\left(b\eta^{-1},a\eta\right)E\left( \frac{a\beta\eta^2}{b} - \eta \right)E\left( u_2(h) - \eta^{-1} \right)
E(u_3(h)) \cdots E(u_{k}(h)), & \text{if } \eta \neq 0, \\
\diag(-b,-a)E\left( u_2(h) + \tfrac{a\beta}{b} \right)E(u_3(h)) \cdots E(u_{k}(h)), & \text{if } \eta = 0.
\end{cases}
\end{align*}

Thus it suffices to treat the case
$$K(h)=\diag(a,b)\,E(\omega)\,\left( \prod_{i=1}^k E(u_i(h))\right), \qquad u_1(h),\ldots,u_k(h)\in\C[h]\setminus\C,\, \omega\in\C.$$ 

From Lemma~\ref{lemmaNS-S}(i), it follows that $k\geq1$. 
If $k=1$, then Lemma~\ref{lemmaNS-S}(ii) forces $\omega\neq0$, and the proof of Lemma~\ref{lemmaNS-S}(iii) yields the claim. If $k\geq2$, then
\begin{align*}
K(h)&\sim_{\sigma^{-1}}  E(u_k(h-1)) \diag(a,b) E(\omega)\left( \prod_{i=1}^{k-1} E(u_i(h))\right) \\ \notag
&= \diag(b,a) E\left(u_k(h-1)\frac{a}{b}\right)E(\omega)\left( \prod_{i=1}^{k-1} E(u_i(h))\right) \\
&= \begin{cases}
\diag\left(b\omega^{-1}, a\omega\right) E\left( \frac{a\omega^2}{b}u_k(h-1) - \omega \right)E\left(u_1(h)-\omega^{-1}\right)\\
\hspace{8cm}\cdot E(u_2(h))\cdots E(u_{k-1}(h)), & \text{if } \omega \neq 0, \\
\diag(-b,-a)E\left( \frac{a}{b}u_k(h-1) +u_1(h) \right)E(u_2(h))\cdots E(u_{k-1}(h)) , & \text{if }\omega =0.
\end{cases}
\end{align*}

If $\omega=0$ and $\frac{a}{b}u_k(h-1)+u_1(h)\in\C$, we iterate the above length-reduction procedure and deduce the claim.
\end{proof}

Alternatively, we now give explicit standard forms for matrices outside $\mathcal S$. To this end, we begin with the following lemma.

\begin{lemma}\label{describsig}
Let $a,b\in\C^*$. A matrix $K(h)\in\GL_2(\C[h])$ is $\sigma^{-1}$-similar to $\diag(a,b)$ if and only if $K(h) = \diag(a,b)$ or
$$
\begin{aligned}
K(h)
&= (-1)^k\,\diag(a,b)_{[k]}\,
   E(0)\,E\left(-\Bigl(\tfrac{a}{b}\Bigr)^{(-1)^k}u_k\right)\cdots
   E\left(-\tfrac{a}{b}\,u_2\right)\\
&\hspace{5cm}\;
   \cdot E\left(-\tfrac{b}{a}\,u_1+u_1(h+1)\right)\,
   E\big(u_2(h+1)\big)\cdots E\big(u_k(h+1)\big),
\end{aligned}
$$
for some $k\in\Z_{\geq 1}$ and $u_i\in\C[h]$ for $1\leq i\leq k$, with
$u_j(h)\notin\C$ for $1<j<k$. If $k=2$, we additionally assume $(u_1,u_2)\neq(0,0)$. Here
$$
\diag(a,b)_{[k]}:=
\begin{cases}
\diag(a,b), & k\ \text{even},\\
\diag(b,a), & k\ \text{odd}.
\end{cases}
$$
\end{lemma}

\begin{proof}
The ``if" part follows by repeatedly applying the length-reduction procedure described in the proof of Theorem \ref{mainstructureS}. To prove the ``only if" part, assume $K(h) \sim_{\sigma^{-1}} \diag(a,b)$, then there exists $P(h) \in \GL_2(\C[h])$ such that $K(h) = P^{-1}(h) \diag(a,b) P(h+1)$.  By Theorem~\ref{standardform}, $P(h)$ admits a unique standard form 
$$P(h) = \diag(\rho_1, \rho_2)\qquad \text{or}\qquad P(h)=\mathbf{E}_{(\rho_1,\rho_2)}(u_1,\dots,u_k),
$$ 
where $\rho_1,\rho_2 \in\C^*$, $k \in \Z_{\geq 1}$, $u_i(h)\in\C[h]$ for $1\leq i\leq k$, with $u_j(h)\notin\C$ for $1<j<k$. If $P(h) =\diag(\rho_1, \rho_2)$, then $K(h) = \diag(a,b)$. Otherwise,
\begin{align*}
K(h)=&\left(\prod_{i=1}^k E\bigl(u_i(h)\bigr)\right)^{-1} \diag(\rho_1^{-1},\rho_2^{-1}) \diag(a,b) \diag(\rho_1 ,\rho_2)\,\sigma^{-1}\left(\prod_{i=1}^k E\bigl(u_i(h)\bigr)\right)\\
=&\, (-1)^{k-1} E(0) \left(\mathop{\overleftarrow{\prod}}_{j=1}^{\,k}E(-u_j(h))\right)  E(0) \diag(a,b) \, \sigma^{-1}\left(\prod_{i=1}^k E\bigl(u_i(h)\bigr)\right) \\
=&\, (-1)^{k-1} E(0) \left(\mathop{\overleftarrow{\prod}}_{j=1}^{\,k}E(-u_j(h))\right) \diag(b,a) E(0)\, \sigma^{-1}\left(\prod_{i=1}^k E\bigl(u_i(h)\bigr)\right) \\
=&\, (-1)^{k-1} E(0) E(-u_k(h)) \cdots E(-u_{2}(h)) \diag(a,b)\\ 
&\hspace{4cm}\cdot E\left(-\tfrac{b}{a}u_1(h)\right) E(0) E(u_1(h+1)) E(u_2(h+1)) \cdots E(u_k(h+1)) \\
=&\, (-1)^k \diag(a,b)_{[k]} E(0) E\left(-\left(\tfrac{a}{b}\right)^{(-1)^k} u_k(h)\right) \cdots E\left(-\tfrac{a}{b} u_2(h)\right) \\
&\hspace{5cm} \cdot E\left(-\tfrac{b}{a} u_1(h) + u_1(h+1)\right) 
E(u_2(h+1)) \cdots E(u_k(h+1)).
\end{align*}
\end{proof} 

\begin{remark}
The presentation of $K(h)$ given in Lemma~\ref{describsig} is not necessarily in standard form (Definition~\ref{standardformdef}). For instance, if $u_1(h)=0$ with $k\geq 1$, then $K(h)$ is not in standard form.
\end{remark}

To present all possible standard forms of $K(h)$, we introduce the following definition.

\begin{definition}Fix $a,b\in\C^*$. Let $\varepsilon,\delta\in\{0,1\}$, $k\geq1$, $\beta,\eta\in\C^*$,
and $\mathbf u=(u_1,\dots,u_k)\in(\C[h])^k$.
Define $\mathbf{D}_{(a,b)}(\varepsilon,\delta;\mathbf u)$, $\mathbf{G}_{(a,b)}(\varepsilon;\beta;\mathbf u)$,
$\mathbf{K}_{(a,b)}(\varepsilon;\eta;\mathbf u)$, and $\mathbf{Z}_{(a,b)}(\beta,\eta;\mathbf u)$ by
$$ 
\begin{aligned}
\mathbf{D}_{(a,b)}(\varepsilon,\delta;\mathbf u)&:= (-1)^k \diag(a,b)_{[k+\varepsilon+\delta]}E(0)^{1-\varepsilon}\Biggl(\mathop{\overleftarrow{\prod}}_{i=2}^{\,k} E\Bigl(-\bigl(\tfrac{a}{b}\bigr)^{(-1)^{i+\delta}}\,u_i\Bigr)\Biggr)E\bigl(u_1\bigr)\\
&\hspace{8cm}\cdot \Biggl(\prod_{j=2}^{\,k} E\bigl(u_j(h+1)\bigr)\Biggr)E(0)^{\varepsilon},
\end{aligned}
$$
$$
\begin{aligned}
\mathbf{G}_{(a,b)}(\varepsilon;\beta;\mathbf u) &:= (-1)^{k+1} \diag(a,b)_{[k+1-\varepsilon]}\diag\left(\beta^{(-1)^{k+1-\varepsilon}}, \beta^{(-1)^{k+\varepsilon}}\right)E(0)^{1-\varepsilon}\\
&\hspace{1.5cm}\cdot \left(\mathop{\overleftarrow{\prod}}_{i=2}^{\,k}E\left(-\left(\beta^{2}\tfrac{a}{b}\right)^{(-1)^{i+1}}u_{i}\right)\right) E\bigl(-\beta^2 \tfrac{a}{b}u_1 -\beta \bigr)E\bigl(u_1(h+1) -\beta^{-1} \bigr)\\
&\hspace{8cm}\cdot\Biggl(\prod_{j=2}^{\,k} E\bigl(u_j(h+1)\bigr)\Biggr)E(0)^{\varepsilon},
\end{aligned}
 $$
 $$
 \begin{aligned}
\mathbf{K}_{(a,b)}(\varepsilon;\eta;\mathbf u)&:= (-1)^{k+1} \diag(a,b)_{[k+1-\varepsilon]}\diag\left(\eta^{-1}, \eta\right) E(-\eta) E\left(-\left(\tfrac{a}{b}\right)^{(-1)^{k+\varepsilon}}u_{k} - \eta^{-1}\right)\\
&\hspace{0.5cm}\cdot \left(\mathop{\overleftarrow{\prod}}_{i=2}^{\,k-1}E\left(-\left(\tfrac{a}{b}\right)^{(-1)^{i+\varepsilon}}u_{i}\right)\right)E\bigl(u_1\bigr)
\Biggl(\prod_{j=2}^{\,k} E\bigl(u_j(h+1)\bigr)\Biggr) E\left(-\eta\left(\tfrac{b}{a}\right)^{(-1)^{k+1-\varepsilon}} \right),
 \end{aligned}
 $$
$$
\begin{aligned}
\mathbf{Z}_{(a,b)}(\beta,\eta;\mathbf u)
&:= (-1)^{k}\diag(a,b)_{[k]} \diag\left(\eta^{-1}\beta^{(-1)^{k}},\eta{\beta}^{(-1)^{k+1}}\right)
E(-\eta)\\
&\hspace{1cm}\cdot E\left(-\left(\beta^{2}\tfrac{a}{b}\right)^{(-1)^{k+1}}u_{k}-\eta^{-1}\right) E\left(-\left(\beta^{2}\tfrac{a}{b}\right)^{(-1)^{k}}u_{k-1}\right)\cdots E\left(-\beta^{-2}\tfrac{b}{a}u_2\right)\\
&\hspace{1cm}\cdot E\left(-\beta^{2}\tfrac{a}{b}\,u_1-\beta\right)E\left(u_1(h+1)-\beta^{-1}\right)E\left(u_2(h+1)\right)\cdots E\left(u_{k}(h+1)\right) \\
&\hspace{9cm}\cdot E\left(-\eta\left(\beta^2 \tfrac{a}{b}\right)^{(-1)^{k+1}}\right).
\end{aligned}
$$

The family $\mathbf D$ is the generic one: after the diagonal factor, there is an $E(0)$ term on exactly one boundary, and the core is diagonally $\sigma$-symmetric in the sense that the right block is obtained from the left block by applying the shift $h\mapsto h+1$ together with a rescaling determined by the diagonal part of $\mathbf D$.
The remaining three families, $\mathbf G,\mathbf K,$ and $\mathbf Z$, follow the same pattern but require additional corrections: a diagonal correction, a correction in the central part, or corrections in the boundary factors of the core.
Table~\ref{tab:DGKZ-features} summarizes these corrections.

\end{definition}

\begin{table}[h] 
\centering
\small
\setlength{\tabcolsep}{6pt}
\renewcommand{\arraystretch}{1.25}
\begin{tabular}{|c|p{0.18\textwidth}|p{0.07\textwidth}|p{0.25\textwidth}|p{0.28\textwidth}|}
\hline
\textbf{Family (length)} & \textbf{Diagonal part} & \textbf{Center} & \textbf{Left boundary} & \textbf{Right boundary}\\
\hline
$\mathbf{D}$ (even)
& generic
& generic
& generic
& generic
\\
\hline
$\mathbf{G}$ (odd)
& $\beta$-correction
& $\beta$-shift
& generic
& generic
\\
\hline
$\mathbf{K}$ (odd)
& $\eta$-correction
& generic
& begins with $E(-\eta)$, then $\eta^{-1}$-shift
& ends with \tiny{$E\left(-\eta\left(\tfrac ba\right)^{(-1)^{k+1-\varepsilon}}\right)$}
\\
\hline
$\mathbf{Z}$ (even)
& $(\beta, \eta)$-correction
& $\beta$-shift
& begins with $E(-\eta)$,  then $\eta^{-1}$-shift
& ends with \tiny{$E\left(-\eta\left(\beta^2\tfrac ab\right)^{(-1)^{k+1}}\right)$}
\\
\hline
\end{tabular}
\smallskip
\caption{Quick-reference features of the families $\mathbf{D},\mathbf{G},\mathbf{K},\mathbf{Z}$.}
\label{tab:DGKZ-features}
\end{table}

\begin{remark}
To avoid any confusion, we adopt the following convention when $k=1$. 
$$\mathbf{K}_{(a,b)}(\varepsilon;\eta;\mathbf u)= \diag(a,b)_{[\varepsilon]}\diag\left(\eta^{-1}, \eta\right) E(-\eta)E\bigl(u_1\bigr) E\left(-\eta\left(\tfrac{b}{a}\right)^{(-1)^{\varepsilon}} \right),$$
\begin{multline*}
\mathbf{Z}_{(a,b)}(\beta,\eta;\mathbf u)=-\diag(b\eta^{-1}\beta^{-1},a\eta\beta)
E(-\eta)E\left(-\beta^{2}\tfrac{a}{b}\,u_1-\beta\right)\\E\left(u_1(h+1)-\beta^{-1}\right)E\left(-\eta\beta^2 \tfrac{a}{b}\right).
\end{multline*}

\end{remark}

\begin{theorem} \label{explicitdescription} Let $a,b \in\C^*$. Then $K(h)\sim_{\sigma^{-1}}\diag(a,b)$ if and only if $K(h)$ is in any of the following standard forms:
$$\diag(a,b)_{[\varepsilon]} ,\quad \diag(-a,-b)_{[\delta]}E(0)^{1-\varepsilon}E(\beta)E(0)^{\varepsilon},$$
$$-\diag\left(a\beta^{-1}, b\beta\right)E(-\beta - \beta\tfrac{b}{a}), \quad \diag\left(-a\beta\eta^{-1},-b\eta\beta^{-1}\right)
E(\eta)\,
E\left(-\frac{(\beta-\eta)b}{a\beta^{2}}+\frac{\beta-\eta}{\beta\eta}\right),$$
$$-\diag\left(b(\beta\eta)^{-1},a\beta\eta\right)E(-\eta)E\left(-\beta^{2}\tfrac{a}{b}\,u_1-\beta-\eta^{-1}\right)E\left(u_1(h+1)-\beta^{-1}\right)E\left(-\eta\beta^2 \tfrac{a}{b}\right),$$
$$\mathbf{D}_{(a,b)}(\varepsilon,\delta;\mathbf u),\quad  \mathbf{G}_{(a,b)}(\varepsilon;\beta;\mathbf u),\quad
\mathbf{K}_{(a,b)}(\varepsilon;\eta;\mathbf u), \quad \mathbf{Z}_{(a,b)}(\beta,\eta;\mathbf u),$$
where $\varepsilon, \delta \in \{0,1\}$, $\beta, \eta \in\C^*$, and $\mathbf u=(u_1,\dots,u_k)\in(\C[h]\setminus \C)^k$ for $k \in \Z_{\geq1}$.
\end{theorem}

\begin{proof}
We defer the case-by-case proof to Appendix A.
\end{proof}

\begin{remark}
The main importance of Theorem~\ref{explicitdescription} concerns the classification of
nonsimple rank-2 modules (more precisely, in describing when $M(\alpha,\mathbf a,K(h))$ admits a proper
nonzero rank-$1$ submodule). On the other hand,  the theorem also gives an explicit description of a fixed class in the nonabelian cohomology set
$H^1(\Z,\GL_2(\C[h]))$.
Fix $D=\diag(a,b)$ with $a,b\in\C^*$, and let $c_D:\Z\to \GL_2(\C[h])$ be the $1$--cocycle with
$c_D(1)=D$. Under the identification in Remark~\ref{rem-def-h1}, the set of all
$K(h)\in\GL_2(\C[h])$ satisfying $K\sim_{\sigma^{-1}}D$ is exactly the set of cocycle representatives
of the class $[c_D]$.
\end{remark}

\section{A family of $U(\mathfrak{h})$-free modules of rank 2} \label{newfamilysection}
In this section, we classify (up to isomorphism) all simple scalar-type modules of rank 2; that is, all modules of the form 
$$M\left(\alpha, \mathbf a, K(h)\right) \quad  \text{where}\;\;\mathbf a\in\{(2,0,0),(0,2,0),(0,0,2)\},\
\alpha\in\C,\
K(h)\in\GL_2(\C[h]).$$

\subsection{Simplicity criteria}
\begin{lemma}\label{submodlem}
Let $\alpha\in 1+\frac{1}{2}\Z_{\geq0}$. Then  $u_{\alpha}(h):=\prod_{j=0}^{2\alpha-2}(h+\alpha-1-j)$ is the unique monic solution of the equation \begin{equation} \label{eq-shift}
(h+\alpha-1)\,q(h-1)=(h-\alpha)\,q(h).
\end{equation}
Furthermore,  $V=\left(u_{\alpha}(h)\C[h]\right)^{\oplus2}$ is a proper rank-$2$ submodule of 
$M\left(\alpha, (0,2,0), K(h)\right)$.
\end{lemma}
\begin{proof} The fact that $q=u_{\alpha}$ is the only monic  solution of \eqref{eq-shift} can be easily proven after  substituting $h$ successively with $\alpha, \alpha-1,...,2-\alpha$ to obtain $q(\alpha-1)=\cdots = q(1-\alpha) =0$. Then, for any $g(h) \in \C[h]$,
$$(h+\alpha-1)\sigma\left(u_{\alpha}(h)g(h)\right)
=u_{\alpha}(h)(h-\alpha)\sigma(g(h)),$$
$$ (h - \alpha + 1)\sigma^{-1}\left(u_{\alpha}(h) g(h) \right)
=u_{\alpha}(h) (h + \alpha)\sigma^{-1}(g(h)).$$

Then for $g_1,g_2\in\C[h]$ the $\mathfrak{sl}(2)$-action on $V$ is:
\begin{equation} \label{submod-calculation}
\begin{aligned}
e \cdot \begin{pmatrix}u_{\alpha}(h)g_1 \\u_{\alpha}(h) g_2 \end{pmatrix}&= u_{\alpha}(h) \sigma\left( K^{-1}(h)\begin{pmatrix}-(h - \alpha + 1) & 0 \\0 & -(h - \alpha + 1)\end{pmatrix} \right)\begin{pmatrix}\sigma(g_1) \\\sigma(g_2)\end{pmatrix}, \\
f \cdot \begin{pmatrix}u_{\alpha}(h)g_1 \\u_{\alpha}(h) g_2 \end{pmatrix}&=u_{\alpha}(h)\begin{pmatrix}h + \alpha & 0 \\0 & h + \alpha\end{pmatrix}K(h)\begin{pmatrix}\sigma^{-1}(g_1) \\\sigma^{-1}(g_2)\end{pmatrix}.
\end{aligned}
\end{equation}

This proves the lemma.
\end{proof}

\begin{proposition}\label{submodesofrank2}
Let $K(h)\in\mathcal S$. Then, for every $\alpha\in\C$,
$$
M\bigl(\alpha, (2,0,0), K(h)\bigr) \quad\text{and}\quad M\bigl(\alpha, (0,0,2), K(h)\bigr)
$$
have no proper rank-$2$ submodules. The same conclusion holds for
$M\bigl(P_{((0,2,0),\alpha)}(h),K(h)\bigr)$ whenever $\alpha\notin 1+\tfrac{1}{2}\Z_{\ge0}$.
\end{proposition}

\begin{proof}
Let $\mathbf a\in\{(2,0,0),(0,2,0),(0,0,2)\}$, and suppose $M\subset M\left(\alpha, \mathbf a, K(h)\right)$ is a proper submodule of rank $2$. Then $M$ has central character $(2\alpha-1)^2$. Consequently, there exist 
$K'(h)\in\GL_2(\C[h])$  and $\beta\in\{\alpha,\,1-\alpha\}$ such that
$$
M \;\simeq\; M\bigl(\beta, \mathbf b, K'(h)\bigr),
$$
where
$$
\mathbf b\in\{(2,0,0),(0,2,0),(0,0,2),(1,1,0),(1,0,1),(0,1,1)\}.
$$

For brevity, write $P_{(\mathbf b,\beta)}(h)=\diag\bigl(\mu_1(h),\mu_2(h)\bigr)$, where $\mu_1(h),\mu_2(h)$ are monic polynomials determined by $\mathbf b$ as in Definition~\ref{tripledef}. The inclusion $M\subset M\left(\alpha, \mathbf a, K(h)\right)$ gives rise to a $ U(\mathfrak{sl}(2))$-homomorphism $\Phi:M\bigl(\beta, \mathbf b, K'(h)\bigr)\to M\left(\alpha, \mathbf a, K(h)\right)$. Then there exists $A(h) \in \Mat_2(\mathbb{C}[h]) $ with $ \det A(h) \neq 0 $ such that, for all $g_1(h),g_2(h)\in\C[h]$,
$$
\Phi \begin{pmatrix}g_1(h) \\ g_2(h) \end{pmatrix} = A(h)\,\begin{pmatrix}g_1(h) \\ g_2(h) \end{pmatrix}.$$

The $f$–action yields the intertwining identity:
$$
P_{(\mathbf{a},\alpha)}(h) K(h) A(h+1) = A(h) P_{(\mathbf{b},\beta)}(h) K'(h),
$$
and taking determinants gives
\begin{equation}\label{deteq}
\det(P_{(\mathbf{a},\alpha)}(h)) \det(K(h)) \det(A(h+1)) = \det(A(h)) \det(P_{(\mathbf{b},\beta)}(h)) \det(K'(h)).
\end{equation}

\noindent
\textbf{Case 1: $\mathbf{a} = (2,0,0)$.}  
Equation \eqref{deteq} becomes
$$
\det K(h)\,\det A(h+1)=\det A(h)\,\mu_1(h)\mu_2(h)\,\det K'(h).
$$

Since $\det K(h),\det K'(h)\in\C^*$ and $\det A(h)\neq0$,$\,\,\det A(h+1)=\mu_1(h)\mu_2(h)\,\det A(h).$
With $\mu_1,\mu_2$ monic, degree comparison forces $\mu_1=\mu_2=1$, hence
$\det A(h)\in\C^*$. Therefore, $\Phi$ is an isomorphism, contradicting that $M$ is proper.

\noindent
\textbf{Case 2: $ \mathbf{a} = (0,0,2)$.}  Equation~\eqref{deteq} becomes
$$
\bigl((h-\alpha+1)(h+\alpha)\bigr)^{2}\det K(h)\,\det A(h{+}1)
=\det A(h)\,\mu_1(h)\mu_2(h)\,\det K'(h).
$$

Since $\det K(h),\det K'(h)\in\C^*$ and $\det A(h)\neq0$,
$$
\bigl((h-\alpha+1)(h+\alpha)\bigr)^{2}\det A(h{+}1)
=\mu_1(h)\mu_2(h)\,\det A(h).
$$

With $\mu_1,\mu_2$ monic, degree and leading–term comparison yield
$\mu_1(h)=\mu_2(h)=(h-\alpha+1)(h+\alpha)$ and hence
$\det A(h)\in\C^*$. Therefore $\Phi$ is an isomorphism, contradicting that $M$ is proper.

\noindent
\textbf{Case 3: $ \mathbf{a} = (0,2,0)$.}  Equation~\eqref{deteq} becomes
$$
(h-\alpha+1)^2\,\det K(h)\,\det A(h+1)=\det A(h)\,\mu_1(h)\mu_2(h)\,\det K'(h),
$$
and since $\det K(h),\det K'(h)\in\C^*$ and $\det A(h)\neq0$, 
$$(h-\alpha+1)^2\,\det A(h+1)=\mu_1(h)\mu_2(h)\,\det A(h).$$ 

It follows that $\det A(h)$ is a product of consecutive linear factors. In particular, there exist $a\in\C,\ k,\ell\in\Z_{\geq0}$ such that
$$
\det A(h)=\prod_{j=0}^{k}\bigl(h-a+j\bigr) \quad\text{or}\quad \det A(h)=\prod_{j=0}^{k}\bigl(h-a+j\bigr)\,\prod_{j=0}^{\ell}\bigl(h-a+j\bigr).
$$

Assume $\det A(h)=\prod_{j=0}^{k}(h-a+j)$. Then
$$
(h-\alpha+1)^2 \Bigl(\prod_{j=1}^{k+1}(h-a+j) \Bigr)
=\Bigl(\prod_{j=0}^{k}(h-a+j)\Bigr)\mu_1(h)\mu_2(h),
$$
so after canceling $\prod_{j=1}^{k}(h-a+j)$, we get
$$
(h-\alpha+1)^2\,(h-a+k+1)=(h-a)\,\mu_1(h)\mu_2(h).
$$

Hence $a=\alpha-1$ and $\mu_1(h)\mu_2(h)=(h-\alpha+k+2)(h-\alpha+1).$
Using $\mu_1\mid\mu_2\mid(h+\alpha)(h-\alpha+1)$ and that $\mu_1,\mu_2$ are monic, the possibilities are
$$
\mu_1(h)=1,\quad \mu_2(h)=(h-\alpha+k+2)(h-\alpha+1),
\quad\text{or}\quad
\mu_1(h)=h-\alpha+k+2,\quad \mu_2(h)=h-\alpha+1.
$$

If $\mu_1(h)=h-\alpha+k+2$ and $\mu_2(h)=h-\alpha+1$, then the condition $\mu_1\mid\mu_2$ forces
$h-\alpha+k+2 = h-\alpha+1$, hence $k+2=1$, a contradiction since $k\geq0$. If $\mu_1(h)=1$ and $\mu_2(h)=(h-\alpha+k+2)(h-\alpha+1)$, the divisibility
$\mu_2\mid(h+\alpha)(h-\alpha+1)$ implies $h-\alpha+k+2\mid h+\alpha$, hence
$h-\alpha+k+2=h+\alpha$ and so 
$$
\alpha=\tfrac{k+2}{2}\in 1+\tfrac{1}{2}\Z_{\geq0}.
$$

Assume $A(h)=\prod_{j=0}^{k}(h-a+j)\prod_{j=0}^{\ell}(h-a+j)$. Then
$$
(h-\alpha+1)^2\prod_{j=1}^{k+1}(h-a+j)\prod_{j=1}^{\ell+1}(h-a+j)
=\prod_{j=0}^{k}(h-a+j)\prod_{j=0}^{\ell}(h-a+j)\mu_1(h)\mu_2(h),
$$
and canceling $\prod_{j=1}^{k}(h-a+j)\prod_{j=1}^{\ell}(h-a+j)$ yields
$$
(h-\alpha+1)^2\,(h-a+k+1)(h-a+\ell+1)=(h-a)^2\,\mu_1(h)\mu_2(h).
$$

Hence $a=\alpha-1$ and $\mu_1(h)\mu_2(h)=(h-\alpha+k+2)(h-\alpha+\ell+2).$
Using $\mu_1\mid\mu_2\mid(h+\alpha)(h-\alpha+1)$ and that $\mu_1,\mu_2$ are monic, we have two possibilities:
$$
\mu_1(h)=1,\qquad \mu_2(h)=(h-\alpha+k+2)(h-\alpha+\ell+2),
$$
or
$$
\mu_1(h)=h-\alpha+k+2,\qquad \mu_2(h)=h-\alpha+\ell+2.
$$

If $\mu_1(h)=1$ and $\mu_2(h)=(h-\alpha+k+2)(h-\alpha+\ell+2)$, then
$$
\mu_2(h)=(h-\alpha+k+2)(h-\alpha+\ell+2)=(h+\alpha)(h-\alpha+1).
$$

Thus $\{k+2,\ell+2\}=\{2\alpha,\,1\}$. Since $k,\ell\in\Z_{\ge0}$, this would force
$\ell+2=1$ (or $k+2=1$), i.e. $\ell=-1$ (or $k=-1$), a contradiction.

If $\mu_1(h)=\mu_2(h)=h-\alpha+k+2=h-\alpha+\ell+2$, then $k=\ell$ and the
divisibility $\mu_2\mid (h+\alpha)(h-\alpha+1)$ implies
$h-\alpha+k+2=h+\alpha$, and
$$
\alpha=\tfrac{k+2}{2}\in 1+\tfrac{1}{2}\Z_{\geq0}.
$$

It remains to consider the case $\mathbf a=(0,2,0)$ with $\alpha\in 1+\frac{1}{2}\Z_{\geq 0}$. By Lemma \ref{submodlem}, the module 
$M\left(\alpha, (0,2,0), K(h)\right)$ contains a nonzero proper submodule of rank $2$. This completes the proof of the proposition.
\end{proof}

\begin{theorem}\label{newfamily}
Let ${\bf{a}} \in \{(2,0,0), (0,2,0), (0,0,2) \}$. The module $M\left(\alpha, \mathbf a, K(h)\right) $ is simple if and only if one of the following holds: 
\begin{enumerate}
    \item $K(h) \in \mathcal S$ and $ \alpha \in \mathbb{C}$ (arbitrary), if ${\bf{a}} = (2,0,0)$ or ${\bf{a}} = (0,0,2)$;
    \item $K(h) \in \mathcal S$ and $ \alpha \in \mathbb{C} \setminus \left( \tfrac{1}{2} \mathbb{Z}_{\geq 0} + 1 \right) $, if ${\bf{a}} = (0,2,0)$.
\end{enumerate}
\end{theorem}

\begin{proof}
Since $\C[h]$ is a PID, every submodule of $M \in \mathcal M(2)$ is a $U(\mathfrak{h})$-free module of rank at most $2$. The theorem follows by combining Proposition \ref{submodesofrank2} and Theorem \ref{mainstructureS}.
\end{proof}

\begin{proposition}
	For $\alpha \in\frac{1}{2} \Z_{\geq 0} +1$, and $K(h) \in \mathcal{S}$,  we have the following non-split short exact sequence:
	$$0 \longrightarrow  M\left(\alpha, (0,-2,0), K(h)\right)  \longrightarrow  M\left(\alpha, (0,2,0), K(h)\right)  \longrightarrow  L(2\alpha-2)^{\oplus 2} \longrightarrow 0$$
	where $  L(2\alpha-2)$ is the simple $(\C h \oplus \C e)$-highest weight module of weight $2\alpha-2$.
\end{proposition}

\begin{proof}
By Lemma \ref{submodlem}, the subspace $V=\left(u_{\alpha}(h)\C[h]\right)^{\oplus2}$, where $u_{\alpha}(h)=\prod_{j=0}^{2\alpha-2}(h+\alpha-1-j)$ is a proper submodule of $M\left(\alpha, (0,2,0), K(h)\right)$. Moreover, \eqref{submod-calculation} shows
$$V \simeq  M\left(1-\alpha, (0,2,0), K(h)\right)= M\left(\alpha, (0,-2,0), K(h)\right).$$

By Theorem \ref{newfamily}, $M\left(1-\alpha, (0,2,0), K(h)\right)$ is simple and hence, the submodule $V$ is simple. Consider the quotient  $Q := M\left(\alpha, (0,2,0), K(h)\right) \big/ V.$
Let 
$$v_1 \;:=\;\begin{pmatrix}\prod_{j=0}^{2\alpha-3}\bigl(h+\alpha-1-j\bigr)\\ 0\end{pmatrix}\;+\; V ,\quad v_2 \;:=\;\begin{pmatrix}0\\ \prod_{j=0}^{2\alpha-3}\bigl(h+\alpha-1-j\bigr)\end{pmatrix}\;+\; V. $$

A direct computation shows 
$$ e\cdot v_i =  0, \quad \text{and} \quad 2h \cdot v_i =  2(\alpha -1)v_i \quad \text{for}\quad i \in \{1,2\}.$$

Therefore, $v_1, v_2$ are highest weight vectors of weight $2(\alpha -1)$. Since $\dim Q = 4\alpha -2$, $Q \simeq L(2\alpha-2)^{\oplus 2}$. This proves the Proposition.
\end{proof}

\subsection{Isomorphism theorem and the group $G$}
In this subsection, we establish the isomorphism theorem for modules of the form
$$
  M\left(\alpha, \mathbf a, K(h)\right),
  \qquad
  \mathbf a\in\{(2,0,0),(0,2,0),(0,0,2)\},\ \alpha\in\C,\ K(h)\in\mathcal S.
$$
In particular, we have an isomorphism theorem for all simple scalar-type rank-2 modules. 

\begin{remark}
By Theorem \ref{mainstructureS}, it suffices to assume that $K(h) =\mathbf{E}_{(a,b)}(u_1,\dots,u_k)$, where  $k \in \Z_{\geq 1}$, $a, b\in\C^*,\;u_1,\dots,u_k \in \C[h]\setminus\C$.
\end{remark}
\begin{definition} \label{nice-action}
Fix $k \in \Z_{\geq 1}$ and set $\mathcal{P}_k:=(\C^*)^2\times(\C[h])^k$. For $(\beta_1,\beta_2,\,u_1,\dots,u_k)\in\mathcal{P}_k$ and $\eta \in \C^*$, define maps $T, S_\eta : \mathcal{P}_k \to \mathcal{P}_k$ by
$$
T\big((\beta_1,\beta_2,\,u_1,\dots,u_k)\big)
\;:=\;
\big(\beta_2,\beta_1, \, \tfrac{\beta_1}{\beta_2}\,\sigma(u_k),\, u_1,\, \dots,\ u_{k-1}\big),
$$
$$S_\eta\big((\beta_1,\beta_2,\,u_1,\dots,u_k)\big):= \big(\beta_1 \eta^{k\bmod 2},\;\beta_2\eta^{-(k\bmod 2)},\;\eta^{\varepsilon_1}u_1,\dots,\eta^{\varepsilon_k}u_k \big),$$
where $\varepsilon_i:=(-1)^{\,i-1+k}$.
\end{definition}

\begin{remark}
It is immediate from the definition that $T$ is bijective. In fact,
$$
T^{-1}\big((\beta_1,\beta_2,\,u_1,\dots,u_k)\big)
\;:=\;
\bigl(\beta_2,\beta_1,\, u_2,\, \dots,\, u_k, \tfrac{\beta_1}{\beta_2} \sigma^{-1}(u_1)\bigr).
$$

Consequently, $T$ generates a $\Z$-action on $\mathcal P_k$ by
$$n\cdot (\mathbf X):=T^{\,n}(\mathbf X),\qquad\text{where}\quad n\in \Z,\;\mathbf X \in \mathcal{P}_k.$$
\end{remark}

For $m\in \Z$, define $\phi_m \in \Aut(\C^*)$ by $\phi_m(\eta) := \eta^{(-1)^{m}}$, for all $\eta\in \C^*$. Then the map
$$\Phi: \Z \to \Aut(\C^*)\qquad\text{where}\qquad m \mapsto \phi_m, $$
is a group homomorphism. 

\begin{lemma} \label{commuting_1}
For all $m\in\Z$ and $\eta\in\C^*$, $T^{m}S_{\eta}T^{-m} = S_{\phi_{m}(\eta)}$.
\end{lemma}

\begin{proof}
The case $m=0$ is immediate. For $m\geq 1$, we proceed by induction on $m$. For $m =1$, let $(\beta_1,\beta_2,\,u_1,\dots,u_k) \in \mathcal{P}_k$, then
$$\begin{aligned}
TS_{\eta}T^{-1} \big( (\beta_1,\beta_2,\,u_1,\dots,u_k)\big) &=TS_{\eta} \left(\bigl(\beta_2,\beta_1,\, u_2,\, \dots,\, u_k, \tfrac{\beta_1}{\beta_2} \sigma^{-1}(u_1)\bigr) \right)\\
&=T \left(\bigl(\beta_2 \eta^{k\bmod 2}, \beta_1 \eta^{-k\bmod 2}, \, \eta^{\varepsilon_1} u_2,\, \dots,\, \eta^{\varepsilon_{k-1}}u_k, \eta^{\varepsilon_k} \tfrac{\beta_1}{\beta_2} \sigma^{-1}(u_1)\bigr) \right) \\
&= \bigl(\beta_1\eta^{-k\bmod 2}, \beta_2 \eta^{k\bmod 2} ,\ \eta^{\varepsilon_k}u_1 , \, \eta^{\varepsilon_1} u_2,\, \dots,\, \eta^{\varepsilon_{k-1}}u_k \bigr).
\end{aligned}$$

Since $\varepsilon_k=-\varepsilon_1$ and, for $2\leq i\leq k$, $\varepsilon_{i-1}=-\varepsilon_i$, then 
$$\begin{aligned}
TS_{\eta}T^{-1} \big( (\beta_1,\beta_2,\,u_1,\dots,u_k)\big) &= \left(\beta_1 \left(\eta^{-1}\right)^{k\bmod 2}, \beta_2 \left(\eta^{-1}\right)^{-k\bmod 2} , \, \left(\eta^{-1}\right)^{\varepsilon_1}u_1,\dots,\left(\eta^{-1}\right)^{\varepsilon_k}u_k \right)\\
&= S_{\eta^{-1}}  \big( (\beta_1,\beta_2,\,u_1,\dots,u_k)\big) =  S_{\phi_{1}(\eta)}\big( (\beta_1,\beta_2,\,u_1,\dots,u_k)\big) .
\end{aligned}$$

This proves the base case $m=1$. The induction step is relatively easy: 
$$ T^{m+1}S_{\eta}T^{-(m+1)} = T S_{\phi_{m}(\eta)} T^{-1} =S_{\phi_{1}\left(\phi_{m}(\eta)\right)} = S_{\phi_{m+1}(\eta)}. $$

It remains to consider $m \leq -1$. In this case we apply the above for $(-m) \geq 1$ and obtain 
$$T^{-m}S_{\phi_m(\eta)}T^{m} = S_{\phi_{-m}\left(\phi_m(\eta)\right)} = S_{\eta}. $$
\end{proof}

\begin{proposition}
The following 
 $$ (\eta, m) \star \mathbf X := S_\eta \big( T^m (\mathbf X) \big),\qquad\text{where}\quad
(\eta,m)\in \C^* \rtimes_{\Phi} \Z,\;\; \mathbf{X}\in \mathcal{P}_k,$$
 defines an action of the group  $\C^* \rtimes_{\Phi} \Z$ on $\mathcal{P}_k$.
\end{proposition}

\begin{proof}
The identity element $(1,0)$ acts trivially since $S_1=T^0=\mathrm{id}$. For $(\eta_i,m_i)\in \C^* \rtimes_{\Phi} \Z$ with $i \in \{1,2\}$, applying Lemma~\ref{commuting_1} gives
\begin{multline*}
(\eta_1,m_1)\star\big((\eta_2,m_2)\star \mathbf{X}\big)
= S_{\eta_1}\,T^{m_1}\,S_{\eta_2}\,T^{m_2}\mathbf{X}
= S_{\eta_1}\,S_{\phi_{m_1}(\eta_2)}\,T^{m_1+m_2}\mathbf{X}\\
= \bigl(\eta_1\phi_{m_1}(\eta_2),m_1+m_2\bigr)\star \mathbf{X}.
\end{multline*}

This proves the proposition.
\end{proof}

\begin{definition} We define 
 $G := \C^* \rtimes_{\Phi} \Z$ and $\Orb_{G}(\mathbf X):=\big\{\,(\eta, n) \star \mathbf X \;\mid\; \eta\in\C^*,\ n\in\Z\,\big\}$. 
\end{definition}

\begin{proposition} \label{sigmasimscalartype}
Let $\beta_1, \beta_2, \gamma_1, \gamma_2\in\C^*$, $k,m \in \Z_{\geq 1}$, and $u_1,\dots,u_k, v_1,\dots,v_m \in\C[h]\setminus\C$. Then
$$
\mathbf{E}_{(\beta_1,\beta_2)}(u_1,\dots,u_k)\sim_{\sigma^{-1}}
\mathbf{E}_{(\gamma_1,\gamma_2)}(v_1,\dots,v_m)
$$
if and only if $m=k$, and $(\gamma_1,\gamma_2,\,v_1,\dots,v_k)\in \Orb_{G}\big( \beta_1,\beta_2,\, u_1,\dots,u_k\big)$. 
\end{proposition}

\begin{proof}
By Proposition \ref{sigmashorten},
$$\mathbf{E}_{(\beta_1,\beta_2)}(u_1,\dots,u_k) \sim_{\sigma^{-1}} \mathbf{E}_{(\gamma_1, \gamma_2)}(v_1,\dots,v_m)$$
implies $k=m$.
Then there exists a matrix $P(h)\in \GL_2(\C[h])$ in standard form,
$$
P(h) = \diag(\rho_1, \rho_2)\quad \text{or}\quad P(h)=\mathbf{E}_{(\rho_1,\rho_2)}(w_1,\dots,w_p),\qquad \rho_1,\rho_2 \in\C^*,
$$
with $w_1,w_p\in\C[h]$ and $w_i\in \C[h]\setminus\C$ for $1<i<p$, such that
$$
\mathbf{E}_{(\gamma_1,\gamma_2)}(v_1,\dots,v_k)
= P(h)^{-1}\,\mathbf{E}_{(\beta_1,\beta_2)}(u_1,\dots,u_k)\,P(h+1).
$$

We determine the $w_i(h)$ for which the length is preserved under $\sigma^{-1}$–conjugation by $P(h)$, namely, $ \;\ell\left(P(h)^{-1}\,\mathbf{E}_{(\beta_1,\beta_2)}(u_1,\dots,u_k)\,P(h+1)\right)=k$. Set $\eta:=\tfrac{\rho_2}{\rho_1}$. A direct computation yields
$$
\begin{aligned}
&\diag(\rho_1, \rho_2)^{-1}\mathbf{E}_{(\beta_1,\beta_2)}(u_1,\dots,u_k)\diag(\rho_1, \rho_2)\\
=&\hspace{0.5cm} \mathbf{E}_{\big(\beta_1 \eta^{k\bmod 2},\;\beta_2\eta^{-(k\bmod 2)}\big)} \big(\eta^{\varepsilon_1}u_1,\dots,\eta^{\varepsilon_k}u_k \big)\\
=&\hspace{0.5cm} \begin{cases}
\mathbf{E}_{(\beta_1,\beta_2)}\left(\eta u_1,\;\eta^{-1}u_2,\dots,\;\eta^{-1}u_k\right), & \text{if } k\ \text{even},\\
\mathbf{E}_{\left(\beta_1\eta,\;\beta_2\eta^{-1}\right)}\left( \eta^{-1} u_1,\;\eta u_2,\dots,\;\eta^{-1} u_k\right), & \text{if } k\ \text{odd}.
\end{cases}
\end{aligned}
$$

Hence, the $\sigma^{-1}$-conjugation of $\mathbf{E}_{(\beta_1,\beta_2)}(u_1,\dots,u_k)$ by $\diag(\rho_1, \rho_2)$  is equivalent to applying $S_\eta$ to $\left(\beta_1,\beta_2,\,u_1,\dots,u_k\right)$. We treat the case where $k$ is even; the case where $k$ is odd is analogous.
\begin{multline*}
P(h)^{-1}\mathbf{E}_{(\beta_1,\beta_2)}(u_1,\dots,u_k) P(h+1) = (-1)^p \diag(\beta_1,\beta_2)_{[p]}\,E(0)\left(\mathop{\overleftarrow{\prod}}_{\ell=2}^{\,p} E\left(-\Bigl(\tfrac{\beta_1}{\beta_2}\Bigr)^{(-1)^\ell}w_\ell\right) \right)\\
\cdot E\left(-\tfrac{\beta_2}{\beta_1}w_1+\eta u_1\right)E\left(\eta^{-1}u_2\right)\dots E\left(\eta^{-1} u_k\right)\, \left(\prod_{i=1}^p E(w_i(h+1))\right).
\end{multline*}

Following the same argument as in the proof of Proposition~\ref{sigmashorten}, the length of the matrix $\mathbf{E}_{(\beta_1,\beta_2)}(u_1,\dots,u_k)$ is preserved for exactly two choices of $(w_1,w_2,...,w_p)$, $p \geq 1$: 

\begin{itemize}
\item[\textbf{Choice 1:}] 
$
w_1=\tfrac{\eta\beta_1}{\beta_2}\,u_1,\qquad w_2=\tfrac{\beta_2}{\eta\beta_1}\,u_2,\qquad 
w_3=\begin{cases}\eta\left(\tfrac{\beta_1}{\beta_2}\right)^2 u_1(h+1), & k=2,\\ \tfrac{\eta\beta_1}{\beta_2}\,u_3, & k>3, \end{cases}\quad \dots \quad (p \geq 1);
$

\item[\textbf{Choice 2:}] 
$w_1=0,\quad w_2=-\eta^{-1}u_k(h-1),\quad w_3=-\eta\,u_{k-1}(h-1),\qquad \dots, \qquad w_p=0 \quad (p \geq 3).
$
\end{itemize}

Let $j$ be defined by
$$
j :=
\begin{cases}
-p, & \text{in Choice 1},\\
p-2, & \text{in Choice 2}.
\end{cases}
$$

With the above choices of $\{w_i: 1 \leq i \leq p\}$, one checks directly that the resulting matrix coincides with the image of
$$
\left(\beta_1,\beta_2,\, \eta u_1,\eta^{-1}u_2,\dots,\eta^{-1}u_k\right)
$$
under $T^j$ for some $j\in\mathbb{Z}$; the case $j=0$ corresponds to the case $P(h)=\diag(\rho_1,\rho_2)$. Therefore, 
$$\mathbf{E}_{(\beta_1,\beta_1)}(u_1,\dots,u_k)\sim_{\sigma^{-1}}\mathbf{E}_{(\gamma_1,\gamma_2)}( v_1,\dots,v_k)$$
if and only if there exist $\eta\in\C^*$ and $j\in\Z$ such that 
$$ (\gamma_1,\gamma_2,\,v_1,\dots,v_k) = T^j\big( S_\eta(\beta_1,\beta_2,\, u_1,\dots,u_k)\big).$$
To complete the proof, we note that for $\mathbf X, \mathbf Y \in \mathcal P_k$,
$$\exists\,\eta\in\C^*,\, j\in\Z:\, \mathbf Y = T^j\big( S_\eta(\mathbf X)\big)\quad \Longleftrightarrow \quad  \mathbf Y \in \Orb_G\left(\mathbf X\right).$$
\end{proof}

\begin{theorem} \label{isomorphismtheorem}
Let $c_1,c_2,d_1,d_2\in\C^*$, $k,m\in\Z_{\ge1}$, $u_1,\dots,u_k,v_1,\dots,v_m\in\C[h]\setminus \C$, $\mathbf a, \mathbf b\in \big\{(2,0,0),(0,2,0),(0,0,2)\big\}$, and $\alpha, \beta \in\C$. Then 
$$ M\left(\alpha, \mathbf a, \mathbf{E}_{(c_1,d_1)}(u_1,\dots,u_k)\right) \simeq M\left(\beta, \mathbf b, \mathbf{E}_{(c_2, d_2)}(v_1,\dots,v_m)\right)$$
if and only if
$$
\alpha=\beta,\quad \mathbf a=\mathbf b,\quad m=k,\quad\text{and}\;\;\;(c_2,d_2,\,v_1,\dots,v_k)\in \Orb_{G}\big((c_1,d_1,\;u_1,\dots,u_k)
\big).
$$
\end{theorem}

\begin{proof}
The theorem follows directly from Theorem \ref{isomorphismthm} and Proposition \ref{sigmasimscalartype}.
\end{proof}

\begin{corollary}
Let $\mathbf a  \in \big\{(2,0,0),(0,0,2)\big\}$ and $\alpha\in \C$; or let $\mathbf a = (0,2,0)$ and  $\alpha \in  \C\setminus \left(\tfrac{1}{2}\mathbb{Z}_{\geq 0}+1 \right) $. There is a natural bijection
$$
\left\{\text{isomorphism classes of simple modules}\; \left[M\big(\alpha, \mathbf a, K(h) \big)\right]\right\} 
\;\xleftrightarrow{\text{\scriptsize 1:1}}\;
\bigsqcup_{k\geq 1} \mathcal{P}_k/G_.
$$
\end{corollary}

\subsection{Interpretation of the group $G$ as a subgroup of $\Aut(\GL_2(\C[h]))$}
Let $P_\gamma=\mathrm{diag}(\gamma,1)\in \GL_2(\C)$ and define
\[
T_\gamma:\GL_2(\C[h])\to\GL_2(\C[h]),\quad T_\gamma(A)=P_\gamma A P_\gamma^{-1}
\]
Let $\sigma(A)(h):=A(h-1)$ and $\theta(A):=E(0)AE(0)^{-1}$. Set $S:=\theta\sigma=\sigma\theta$.
Let $H$ be the subgroup of $ \mathrm{Aut}(\GL_2(\C[h]))$ generated by $T_{\gamma}$ ($\gamma \in \C^*$) and $S$. Then one can easily check that The map
\[
\Psi:\ \C^*\rtimes_\Phi\Z\longrightarrow H,\qquad (\gamma,m)\longmapsto T_\gamma S^m,
\]
is an isomorphism.

\section*{Appendix A} \label{app-case-analysis}
Recall the definition of $T_{a,b}$ from Lemma \ref{surj-lem-ab}. The following lemma will be used frequently throughout the proof of Theorem \ref{explicitdescription}.
\begin{lemma} \label{usefullem}
Let $u \in \C[h]\setminus \C$ and $c \in \C^*$. Then $T_{1,c}(u)\neq 0$. Moreover, $T_{1,c}(u)=\beta$ for some $\beta \in \C^*$ if and only if $c=1$ and $u=\beta h+\theta$ for some $\theta \in \C$.
\end{lemma}

\smallskip \noindent
\textbf{Proof of Theorem~\ref{explicitdescription}.}
Throughout the proof, let $k\in\Z_{\geq1}$ be arbitrary and let $u_1,\dots,u_k\in\C[h]\setminus\C$. Set $\mathbf u=(u_1,\dots,u_k)$. We proceed by a case-by-case analysis, together with systematic applications of Lemma~\ref{shortenlem}. From Lemma \ref{describsig}, $K(h) = \diag(a,b)$ or
$$
\begin{aligned}
K(h)&= (-1)^m\,\diag(a,b)_{[m]}\,E(0)\,E\left(-\Bigl(\tfrac{a}{b}\Bigr)^{(-1)^m}v_m\right)\cdots
   E\left(-\tfrac{a}{b}\,v_2\right)\\
&\hspace{5.4cm} \cdot E\left(-\tfrac{b}{a}\,v_1+v_1(h+1)\right)\,
   E\big(v_2(h+1)\big)\cdots E\big(v_m(h+1)\big),
\end{aligned}
$$
for some $m\in\Z_{\geq1}$ and $v_i\in\C[h]$ for $1\leq i\leq m$, with
$v_j\notin\C$ for $1<j<m$. If $m =2$, then $(v_1,v_2) \neq (0,0)$.

\smallskip\noindent
\textbf{Case 1:} $-\tfrac{b}{a}v_1+v_1(h+1) \in\C[h]\setminus\C$ and $v_m\in\C[h]\setminus\C$.

In this case, $K(h)$  is already in standard form and can be expressed as 
$$K(h) =  \mathbf{D}_{(a,b)}(0,0; \mathbf u).$$

\smallskip\noindent
\textbf{Case 2:} $-\tfrac{b}{a}v_1+v_1(h+1)\in\C[h]\setminus\C$ and $v_m=0$ ($m \geq 2$).
$$
\begin{aligned}
K(h)
&= (-1)^{m+1}\diag(a,b)_{[m]}\,
   E\left(-\Bigl(\tfrac{a}{b}\Bigr)^{(-1)^{m-1}}v_{m-1}\right)\cdots
   E\left(-\tfrac{a}{b}\,v_2\right)\\
&\hspace{4cm} \cdot
   E\left(-\tfrac{b}{a}\,v_1+v_1(h+1)\right)\,
   E\big(v_2(h+1)\big)\cdots E\big(v_{m-1}(h+1)\big)\,E(0).
\end{aligned}
$$

Then, $K(h)$ has the standard form $ \mathbf{D}_{(a,b)}(1,0; \mathbf u).$

\smallskip\noindent
\textbf{Case 3.} $-\tfrac{b}{a}v_1+v_1(h+1)\in\C[h]\setminus\C$ and $v_m\in\C^*$  ($m \geq 2$). 

For brevity, set $\eta_1:=-\Bigl(\tfrac{a}{b}\Bigr)^{(-1)^m}v_m \in\C^*.$

\smallskip
\textbf{Case 3.1 $m=2$.} Then 
$$
\begin{aligned}
K(h)&= \diag(a,b) E(0)E(\eta_1)E\left(-\tfrac{b}{a}\,v_1+v_1(h+1)\right)E\left(-\eta_1\tfrac{b}{a}\right)\\
&= \diag(a\eta_1^{-1},b\eta_1) E(-\eta_1) E\left(-\tfrac{b}{a}\,v_1+v_1(h+1) - \eta_1^{-1}\right) E\left(-\eta_1\tfrac{b}{a}\right)
\end{aligned}
$$

\smallskip
\textbf{Case 3.2 $m\geq 3$.} 
$$
\begin{aligned}
K(h)
&=(-1)^m\diag(a,b)_{[m]}\,
   \diag\left(\eta_1^{-1},\eta_1 \right)E(-\eta_1)\,E\left(-\Bigl(\tfrac{a}{b}\Bigr)^{(-1)^{m-1}}v_{m-1}(h)-\eta_1^{-1}\right)\cdots
   E\left(-\tfrac{a}{b}\,v_2(h)\right)\\
&\hspace{2.5cm} \cdot E\left(-\tfrac{b}{a}\,v_1(h)+v_1(h+1)\right)\,
   E\big(v_2(h+1)\big)\cdots E\big(v_{m-1}(h+1)\big)\,
   E\left(-\eta_1\Bigl(\tfrac{b}{a}\Bigr)^{(-1)^m}\right),
\end{aligned}
$$

In both Cases 3.1 and 3.2, $K(h)$ admits the following general standard form
$$K(h) =  \mathbf{K}_{(a,b)}(0;\eta_1; \mathbf u).$$

\smallskip\noindent
\textbf{Case 4.} $-\tfrac{b}{a}v_1+v_1(h+1)=0$ (equivalently, $v_1=0$) and $v_m \in\C[h]\setminus\C$ (if $m\geq 2$) .

\smallskip
\textbf{Case 4.1 $m=1$.} Then, $K(h) = -\diag(b,a)E(0)^2=\diag(b,a).$

\smallskip
\textbf{Case 4.2 $m\geq 2$.} 
$$
\begin{aligned}
K(h)&= (-1)^{m+1}\diag(a,b)_{[m]}\,
   E(0)\,E\left(-\Bigl(\tfrac{a}{b}\Bigr)^{(-1)^m}v_m\right)\cdots
   E\big(-\tfrac{a}{b}\,v_2+v_2(h+1)\big)\\
&\hspace{10cm}\cdot E\big(v_3(h+1)\big)\cdots E\big(v_m(h+1)\big).
\end{aligned}
$$

If $-\tfrac{a}{b}v_2(h)+v_2(h+1)\in\C[h]\setminus\C$, then $K(h)$ is in standard form, and admits the general formula
$$K(h) =  \mathbf{D}_{(a,b)}(0,1; \mathbf u).$$

\smallskip
If $-\tfrac{a}{b}v_2+v_2(h+1)=\beta_2\in\C^*$, then Lemma~\ref{usefullem} forces $a=b$. If $m=2$, we have $K(h) =-a\, E(0)E(\beta_2)$. If $m \geq 3$, then
$$
\begin{aligned}
K(h)&=(-1)^{m+1}a\,E(0)\,E(-v_m)\cdots E(-v_3)\,E(\beta_2)\,
  E\big(v_3(h+1)\big)\cdots E\big(v_m(h+1)\big)\\
&=(-1)^{m+1}a\diag\left(\beta_2^{(-1)^{\,m-1}}, \beta_2^{(-1)^{m}} \right)  E(0)\,E\left(-(\beta_2^2)^{(-1)^{\,m-1}}v_m\right)\cdots
  E\left(-\beta_2^{-2}v_4\right) \\
&\hspace{3.3cm}\cdot E\left(\beta_2^2\left(-v_3-\beta_2^{-1}\right)\right)E\left(v_3(h+1)-\beta_2^{-1}\right)E\big(v_4(h+1)\big)\cdots E\big(v_m(h+1)\big).
\end{aligned}
$$

The standard form of $K(h)$ can be expressed as $\mathbf{G}_{(a,a)}(0;\beta_2; \mathbf u)$.

\smallskip\noindent
\textbf{Case 5.} $-\tfrac{b}{a}v_1+v_1(h+1)=0$ and $v_m=0$ ($m\geq 3$).
$$
\begin{aligned}
K(h)
&= (-1)^{m}\diag(a,b)_{[m]}\,E\left(-\left(\tfrac{a}{b}\right)^{(-1)^{m-1}}v_{m-1}\right)\cdots E\left(-\tfrac{b}{a}v_3\right)\\
&\hspace{4.5cm}\cdot\;
E\left(-\tfrac{a}{b}v_2+v_2(h+1)\right)
E\left(v_3(h+1)\right)\cdots E\left(v_{m-1}(h+1)\right)E(0).
\end{aligned}
$$

If $-\tfrac{a}{b}v_2+v_2(h+1)\in\C[h]\setminus\C$, then $K(h)$ is in standard form and it can be written generally as
$K(h) = \mathbf{D}_{(a,b)}(1,1; \mathbf u).$

\smallskip
If $-\tfrac{a}{b}v_2+v_2(h+1) =\beta_2 \in \C^*$, then Lemma~\ref{usefullem} implies $a=b$. If $m=3$, we deduce that $K(h) =-a\,E(\beta_2)E(0)$. If $m \geq 4$, then
$$
\begin{aligned}
K(h)
&= (-1)^{m}a\, E\left(-v_{m-1}\right)\cdots E\left(-v_3\right)\,E(\beta_2)\,
    E\left(v_3(h+1)\right)\cdots E\left(v_{m-1}(h+1)\right)E(0)\\  
&= (-1)^{m}a\,\diag\left(\beta_2^{(-1)^{\,m-1}}, \beta_2^{(-1)^{m}} \right) E\left(-(\beta_2^2)^{(-1)^{m}} v_{m-1}\right)\cdots
   E\left(-\beta_2^{-2}v_4\right)\\
&\hspace{1cm}\cdot  E\left(\beta_2^2\left(-v_3-\beta_2^{-1}\right)\right) E\left(v_3(h+1)-\beta_2^{-1}\right)E(v_4(h+1))\dots E\left(v_{m-1}(h+1)\right)E(0).
\end{aligned}
$$

The general standard form of $K(h)$ can be written as $\mathbf{G}_{(a,a)}(1;\beta_2; \mathbf u)$.

\smallskip\noindent
\textbf{Case 6.} $-\tfrac{b}{a}v_1+v_1(h+1)=0$ and $v_m\in\C^*$ ($m \geq 2$).

As in Case 3, let $\eta_1=-\Bigl(\tfrac{a}{b}\Bigr)^{(-1)^m}v_m$.

\smallskip
\textbf{Case 6.1 $m=2$.} Then
$$K(h) =  \diag(a,b) E(0)E(\eta_1)E\left(0\right)E\left(-\eta_1\tfrac{b}{a}\right) = \begin{cases} \diag(-a,-b) E(0)E\left(\eta_1-\eta_1\tfrac{b}{a}\right), & \text{if } a\neq b,\\ a \I_2, & \text{if } a= b. \end{cases}  $$

\smallskip
\textbf{Case 6.2 $m=3$.} Then
$$
\begin{aligned}
K(h)&= \diag(b,a) E(0)E(\eta_1)E\left(-\tfrac{a}{b}\,v_2+v_2(h+1)\right)E\left(-\eta_1\tfrac{a}{b}\right)\\
&= \diag(b\eta_1^{-1},a\eta_1) E(-\eta_1) E\left(-\tfrac{a}{b}\,v_2+v_2(h+1) - \eta_1^{-1}\right) E\left(-\eta_1\tfrac{a}{b}\right).
\end{aligned}
$$

If $-\tfrac{a}{b}v_2+v_2(h+1)\in\C[h]\setminus\C$, $K(h)$ is in standard form, which we include as part of Case 6.3. If $-\tfrac{a}{b}v_2+v_2(h+1) = \beta_2 \in \C^*$, then Lemma~\ref{usefullem} implies that $a=b$ and  

$$K(h) = \begin{cases}a\diag(-\eta_1^{-1},-\eta_1) E(-2\eta_1), & \text{if } \beta_2 = \eta_1^{-1},\\ a\diag\big((\eta_1\beta_2 - 1)^{-1}, \eta_1\beta_2 - 1\big)\, E\left(\beta_2-\eta_1\beta_2^2 \right) E\left(-\eta_1 -(\beta_2 - \eta_1^{-1})^{-1}\right), & \text{if } \beta_2 \neq \eta_1^{-1}. \end{cases}$$

\smallskip
\textbf{Case 6.3 $m\geq 4$.} Then,
$$
\begin{aligned}
K(h)&=(-1)^{m+1}\diag(a,b)_{[m]}
 \diag\left(\eta_1^{-1},\eta_1 \right)
E(-\eta_1)\,
E\left(-\Bigl(\tfrac{a}{b}\Bigr)^{(-1)^{m-1}}v_{m-1}-\eta_1^{-1}\right) \\
&\hspace{3.5cm} \cdot E\left(-\Bigl(\tfrac{a}{b}\Bigr)^{(-1)^{m-2}}v_{m-2}\right) \cdots
E\left(-\tfrac{b}{a}\,v_3\right)
E\left(-\tfrac{a}{b}\,v_2+v_2(h+1)\right)\\
&\hspace{6cm} \cdot E\big(v_3(h+1)\big)\cdots E\big(v_{m-1}(h+1)\big)\,
E\left(-\eta_1\Bigl(\tfrac{b}{a}\Bigr)^{(-1)^m}\right).
\end{aligned}
$$

If $-\tfrac{a}{b}v_2+v_2(h+1)\in\C[h]\setminus\C$, then $K(h)$ is in standard form, and it can be written in the general form $K(h) =  \mathbf{K}_{(a,b)}(1;\eta_1; \mathbf u).$

\smallskip
If $-\tfrac{a}{b}v_2+v_2(h+1) =\beta_2\in \C^*$, then $a=b$. If $m =4$, then
$$
\begin{aligned}
K(h)&=-\diag\left(a\eta_1^{-1},a\eta_1 \right)E(-\eta_1)E\left(-v_{3}-\eta_1^{-1}\right)E(\beta_2)E\big(v_3(h+1)\big)\,E(-\eta_1)\\
&=-\diag\left(a\eta_1^{-1}\beta_2,a\eta_1 \beta_2^{-1} \right)E\left(-\eta_1 \beta_2^{-2}\right)E\left(-\beta_2^2(v_{3}+\eta_1^{-1}+ \beta_2^{-1}) \right)E\big(v_3(h+1)-\beta_2^{-1}\big)\,E(-\eta_1).
\end{aligned}
$$

If $m \geq 5$, then 
$$
\begin{aligned}
K(h)&=(-1)^{m+1}  \diag\left(a\eta_1^{-1},a\eta_1 \right)E(-\eta_1)\,
E\left(-v_{m-1}-\eta_1^{-1}\right)E\left(-v_{m-2}\right)\cdots E\left(-v_3\right)\,E(\beta_2)\\
&\hspace{8cm}\cdot E\big(v_3(h+1)\big)\cdots E\big(v_{m-1}(h+1)\big)\,E(-\eta_1)\\
&=(-1)^{m+1} \diag\left(a\eta_1^{-1}\beta_2^{(-1)^{m}},a\eta_1 \beta_2^{(-1)^{m-1}} \right)
E\big(-(\beta_2^2)^{(-1)^{m-1}}\eta_1\big)\,
E\left((\beta_2^2)^{(-1)^{m-2}}\left(-v_{m-1}-\eta_1^{-1}\right)\right)\\
&\hspace{2cm}\cdot E\left(-(\beta_2^2)^{(-1)^{m-3}}v_{m-2}\right)\cdots\,
E\left(-\beta_2^{-2}v_4\right)E\left(\beta_2^2\left(-v_3-\beta_2^{-1}\right)\right)
E\left(v_3(h+1)-\beta_2^{-1}\right)\\
&\hspace{8cm}\cdot E\left(v_4(h+1)\right)\cdots E\big(v_{m-1}(h+1)\big)\,E(-\eta_1).\\
\end{aligned}
$$

The general standard form of $K(h)$ can be written as
$K(h) =  \mathbf{Z}_{(a,a)}\left(\beta_2, \eta_1\left(\beta_2^2\right)^{(-1)^k}; \mathbf u\right).$

\smallskip\noindent
\textbf{Case 7.} $-\tfrac{b}{a}v_1+v_1(h+1)\in\C^*$ and $v_m \in\C[h]\setminus\C$. 

For brevity, set $\beta_1:= -\tfrac{b}{a}v_1+v_1(h+1)$.

\smallskip
\textbf{Case 7.1 $m=1$.}  Then, $K(h)=\diag(-b,-a)E(0)E(\beta_1).$

\smallskip
\textbf{Case 7.2 $m\geq2$.} In this case,
$$
\begin{aligned}
K(h)
&= (-1)^{m}\diag(a,b)_{[m]}
E(0)\,
E\left(-\left(\tfrac{a}{b}\right)^{(-1)^{m}}v_{m}\right)\cdots
E\left(-\tfrac{a}{b}v_2\right)\,E(\beta_1)\\
& \hspace{9.5cm}\cdot E\left(v_2(h+1)\right)\cdots E\left(v_m(h+1)\right)\\
&= (-1)^{m} \diag(a,b)_{[m]} \diag\left(\beta_1^{(-1)^{m}},{\beta_1}^{(-1)^{m-1}}\right)
E(0)\,E\left(-\left(\beta_1^{2}\tfrac{a}{b}\right)^{(-1)^{m}}v_m\right)\cdots
E\left(-\beta_1^{-2}\tfrac{b}{a}v_3\right)\\
&\hspace{3.0cm}\cdot E\left(-\beta_1^{2}\tfrac{a}{b}\,v_2-\beta_1\right)
E\left(v_2(h+1)-\beta_1^{-1}\right)E\left(v_3(h+1)\right)\cdots E\left(v_m(h+1)\right).
\end{aligned}
$$

The general standard form of $K(h)$ is $\mathbf{G}_{(a,b)}(0;\beta_1; \mathbf u)$.

\smallskip\noindent
\textbf{Case 8.} $-\tfrac{b}{a}v_1+v_1(h+1)\in\C^*$ and $v_m=0$ ($m \geq 2$). 

As in Case 7, set $\beta_1= -\tfrac{b}{a}v_1+v_1(h+1)$.

\smallskip
\textbf{Case 8.1 $m=2$.} Then, $K(h)=\diag(-a,-b)E(\beta_1)E(0).$

\smallskip
\textbf{Case 8.2 $m\geq3$.} In this case,
$$
\begin{aligned}
K(h)
&= (-1)^{m-1}\diag(a,b)_{[m]} \diag\left(\beta_1^{(-1)^{m}}, \beta_1^{(-1)^{m-1}}\right)
E\left(-\left(\beta_1^{2}\tfrac{a}{b}\right)^{(-1)^{m-1}}v_{m-1}\right)\cdots E\left(-\beta_1^{-2}\tfrac{b}{a}v_3\right) \\
&\hspace{2cm}\cdot E\left(-\beta_1^{2}\tfrac{a}{b}\,v_2-\beta_1\right) E\left(v_2(h+1)-\beta_1^{-1}\right)E\left(v_3(h+1)\right)\cdots E\left(v_{m-1}(h+1)\right)E(0).
\end{aligned}
$$

The general standard form of $K(h)$ can be written as $ \mathbf{G}_{(a,b)}(1;\beta_1; \mathbf u).$

\smallskip\noindent
\textbf{Case 9.} $-\tfrac{b}{a}v_1+v_1(h+1)\in\C^*$ and $v_m\in\C^*$ ($m \geq 2$).

As in Case 7, set $\beta_1= -\tfrac{b}{a}v_1+v_1(h+1)$ and let $\eta_2:=-\Bigl(\beta^2\tfrac{a}{b}\Bigr)^{(-1)^m}v_m.$

\smallskip
\textbf{Case 9.1 $m=2$.} Then 
$$
\begin{aligned}
K(h)&= \diag(a,b) E(0)E\left(\eta_2\beta_1^{-2}\right)E\left(\beta_1\right)E\left(-\eta_2\beta_1^{-2}\tfrac{b}{a}\right)\\
&= \diag(a\eta_2^{-1}\beta_1^2,b\eta_2 \beta_1^{-2}) E\left(-\eta_2\beta_1^{-2}\right) E\left(\beta_1 - \eta_2^{-1}\beta_1^2\right) E\left(-\eta_2\beta_1^{-2}\tfrac{b}{a}\right).
\end{aligned}
$$

If $\beta_1 = \eta_2$, then $K(h) =  \diag\left(-a\beta_1, -b\beta_1^{-1}\right) E\left(-\beta_1^{-1} -\beta_1^{-1}\tfrac{b}{a} \right)$. If $\beta_1 \neq \eta_2$, then
$$
K(h) =  \diag\left(a\left(\eta_2\beta_1^{-1} - 1\right)^{-1}, b\left(\eta_2\beta_1^{-1} - 1\right)\right)E\left(\beta_1-\eta_2 \right)E\left(-\eta_2\beta_1^{-2}\tfrac{b}{a}-(\beta_1 - \eta_2^{-1}\beta_1^{2})^{-1}\right).
$$

We note that if $\eta_2\beta_1^{-1} - 1=-\tfrac{a}{b}$, then $K(h) = -\diag(b,a) E\left(\beta_1\tfrac{a}{b} \right) E(0) $.

\smallskip
\textbf{Case 9.2 $m\geq 3$.} 
If $m=3$, then
$$
\begin{aligned}
K(h) &= -\diag(b\beta_1^{-1},a\beta_1)E(0)E(\eta_2)E\left(-\beta_1^{2}\tfrac{a}{b}\,v_2-\beta_1\right)E\left(v_2(h+1)-\beta_1^{-1}\right)E\left(-\eta_2\beta_1^2 \tfrac{a}{b}\right)\\
&=-\diag(b(\beta_1\eta_2)^{-1},a\beta_1\eta_2)E(-\eta_2)E\left(-\beta_1^{2}\tfrac{a}{b}\,v_2-\beta_1-\eta_2^{-1}\right)E\left(v_2(h+1)-\beta_1^{-1}\right)E\left(-\eta_2\beta_1^2 \tfrac{a}{b}\right).
\end{aligned}
$$

If $m \geq 4$, then
$$
\begin{aligned}
K(h)
&= (-1)^{m}
\diag(a,b)_{[m]} \diag\left(\beta_1^{(-1)^{m}},{\beta_1}^{(-1)^{m-1}}\right)
E(0)\,
E(\eta_2) E\left(-\left(\beta_1^{2}\tfrac{a}{b}\right)^{(-1)^{m-1}}v_{m-1}\right)\cdots\\
&\hspace{2cm} \cdot E\left(-\beta_1^{-2}\tfrac{b}{a}v_3\right)E\left(-\beta_1^{2}\tfrac{a}{b}\,v_2-\beta_1\right)
E\left(v_2(h+1)-\beta_1^{-1}\right)E\left(v_3(h+1)\right)\\
&\hspace{7cm}\cdots E\left(v_{m-1}(h+1)\right)E\left(-\eta_2\left(\beta_1^2 \tfrac{a}{b}\right)^{(-1)^{m-1}}\right)\\
&= (-1)^{m}\diag(a,b)_{[m]} \diag\left(\eta_2^{-1}\beta_1^{(-1)^{m}},\eta_2{\beta_1}^{(-1)^{m-1}}\right)
E(-\eta_2)\, E\left(-\left(\beta_1^{2}\tfrac{a}{b}\right)^{(-1)^{m-1}}v_{m-1}-\eta_2^{-1}\right)\\
&\hspace{2cm}\cdot E\left(-\left(\beta_1^{2}\tfrac{a}{b}\right)^{(-1)^{m-2}}v_{m-2}\right)\cdots E\left(-\beta_1^{-2}\tfrac{b}{a}v_3\right)E\left(-\beta_1^{2}\tfrac{a}{b}\,v_2-\beta\right)
E\left(v_2(h+1)-\beta_1^{-1}\right)\\
&\hspace{6cm}\cdot E\left(v_3(h+1)\right)\cdots E\left(v_{m-1}(h+1)\right) E\left(-\eta_2\left(\beta_1^2 \tfrac{a}{b}\right)^{(-1)^{m-1}}\right).
\end{aligned}
$$

The general standard form of $K(h)$ can be written as $\mathbf{Z}_{(a,b)}(\beta_1, \eta_2; \mathbf u)$.

\end{document}